\documentclass{amsart}

\usepackage{graphicx,amsmath,amsfonts,amssymb,amsthm,color}
\usepackage[normalem]{ulem}

\usepackage[utf8]{inputenc} 
\usepackage{textcomp} 

\usepackage[pdftex,bookmarks,colorlinks,breaklinks]{hyperref}  

\definecolor{dullmagenta}{rgb}{0.4,0,0.4}   
\definecolor{darkblue}{rgb}{0,0,0.4}
\definecolor{darkkgreen}{rgb}{0,0.7,0}
\definecolor{darkkred}{rgb}{0.7,0,0}

\hypersetup{linkcolor=black,citecolor=black,filecolor=black,urlcolor=black} 

\usepackage[active]{srcltx}

\theoremstyle{plain}
\newtheorem{maintheorem}{Theorem}

\newtheorem{maincorollary}[maintheorem]{Corollary}
\newtheorem{theorem}{Theorem}[section]

\newtheorem*{thm*}{Theorem}

\newtheorem{lemma}[theorem]{Lemma}

\newtheorem{corollary}[theorem]{Corollary}

\newtheorem{proposition}[theorem]{Proposition}

\newtheorem{problem}[theorem]{Problem}

\newtheorem*{prob*}{Problem}

\theoremstyle{definition}
\newtheorem{definition}[theorem]{Definition}

\theoremstyle{remark}
\newtheorem{remark}[theorem]{Remark}
\newtheorem*{rem*}{Remark}
\newtheorem{example}[theorem]{Example}

\newcommand{\claim}{\medbreak {\scshape Claim}. }

\newcommand{\Proof}{\medbreak\noindent {\it Proof}. }
\newcommand{\Proofof}[1]{\medbreak \noindent{\it Proof of #1}. }

\newcommand{\Endproof}{\hfill\qed\medbreak}

\newcommand{\References}[1]{
\begin{thebibliography}{#1} }

\newcommand{\Endrefs}{\end{thebibliography}}

\newcommand{\name}[1]{{\scshape#1}}
\newcommand{\jname}[1]{{\it#1}}

\newcounter{Lcount}

\newcommand{\aA}{{\mathbb A}}
\newcommand{\NN}{{\mathbb N}}
\newcommand{\RR}{{\mathbb R}}

\newcommand{\TT}{{\mathbb T}}
\newcommand{\ZZ}{{\mathbb Z}}
\newcommand{\DD}{{\mathbb D}}

\newcommand{\cA}{{\mathcal A}}
\newcommand{\cB}{{\mathcal B}}

\newcommand{\cD}{{\mathcal D}}

\newcommand{\cG}{{\mathcal G}}
\newcommand{\cH}{{\mathcal H}}
\newcommand{\cI}{{\mathcal I}}
\newcommand{\cJ}{{\mathcal J}}
\newcommand{\cK}{{\mathcal K}}
\newcommand{\cL}{{\mathcal L}}
\newcommand{\cO}{{\mathcal O}}
\newcommand{\cR}{{\mathcal R}}
\newcommand{\cS}{{\mathcal S}}
\newcommand{\cT}{{\mathcal T}}
\newcommand{\cU}{{\mathcal U}}
\newcommand{\cV}{{\mathcal V}}

\newcommand{\tih}{\tilde h}

\newcommand{\tN}{\tilde N}

\newcommand{\biR}{\textbf{\textit{R}}}

\newcommand{\Diff}{{\rm Diff}}
\newcommand{\diff}{{\rm Diff}}
\newcommand{\diam}{{\rm diam}}
\newcommand{\inter}{{\rm int}}
\newcommand{\orbit}{\cO}
\newcommand{\andrm}{~ {\rm and} ~ }

\newcommand{\ifrm}{~ {\rm if} ~ }

\newcommand{\veps}{\varepsilon}
\newcommand{\w}{\omega}
\newcommand{\x}{\times}
\newcommand{\vazio}{\varnothing}

\newcommand{\La}{\Lambda}

\newcommand{\UU}{\mathbb{U}}

\begin{document}

\title{Robust Transitivity in Hamiltonian Dynamics}
\author{Meysam Nassiri $~ ~$  Enrique R. Pujals}
\address{School of Mathematics, Institute for Research in Fundamental Sciences (IPM),  P. O. Box 19395-5746,  Tehran, Iran.}
\email{nassiri@impa.br}
\address{IMPA,  Estrada D. Castorina 110, 22460-320 $~$ Rio de Janeiro, Brazil.}
\email{enrique@impa.br}

\begin{abstract} 
A goal of this work is to study the dynamics in the complement of KAM tori with focus on non-local robust transitivity.
We introduce $C^{r}$ open sets ($r=1, 2, \dots, \infty$) of symplectic diffeomorphisms and Hamiltonian systems,
exhibiting \textit{large} robustly transitive sets. 
We show that the $C^{\infty}$ closure of such open sets contains a variety of systems, including  so-called {\it a priori} unstable integrable systems. 
In addition, the existence of ergodic measures with large support is obtained for all those systems.
A main ingredient of the proof is a combination of studying minimal dynamics of symplectic iterated function systems and a new tool in Hamiltonian dynamics which we call ``symplectic blender".

\vspace{.5cm}
\noindent {\scshape R\'esum\'e.}   {\it  Transitivit\'e robuste en dynamique hamiltonienne}.  Un objectif de ce travail est d'\'etudier la dynamique  sur le compl\'ementaire des tores KAM en mettant l'accent sur la transitivit\'e robuste non locale.
Nous introduisons les ensembles ouverts de diff\'eomorphismes symplectiques et de syst\`emes hamiltoniens, pr\'esentant de grands ensembles robustement transitifs.  L'adh\'erence de ces ensembles ouverts (en topologie $C^r$, $r=1, 2, \dots, \infty$) contient un grand nombre de syst\`emes, y compris les syst\`emes int\'egrables  a priori instables.
En outre, l'existence de mesures ergodiques  avec un grand support est obtenu pour l'ensemble de ces syst\`emes.
L'ingr\'edient principal des preuves est de combiner l'\'etude  de  syst\`emes  it\'er\'es de fonctions  de dynamique minimale  et un nouvel outil de la dynamique hamiltonienne que nous appelons ``m\'elangeurs symplectiques".

\end{abstract}

\maketitle

\setcounter{tocdepth}{1}
\begin{small}
\tableofcontents
\end{small}

\section{Introduction and main results} \label{sec intro}
The theory of Kolmogorov, Arnold and Moser (KAM) gives a precise description of
the dynamics of a set of large measure of orbits for any small
perturbation of a non-degenerate integrable Hamiltonian system. These orbits lie
on the invariant  KAM tori for which the dynamics 
are equivalent to irrational (Diophantine) rotations. 
In the case of autonomous systems in two degrees of freedom 
or time-periodic systems in one degree of freedom (i.e., 1.5 degrees of freedom), the KAM Theorem proves the stability of {\it all} orbits, in the sense that the action variable do not vary
much along the orbits.  This, of course, is not the case if the degree of freedom is larger than two, where the KAM tori has  codimension of at least two. A natural question arises: {\it Do generic perturbations of integrable systems in higher dimensions exhibit instabilities?} 
The first example of instability is due to Arnold \cite{a2}, who constructed a family of small perturbations of a non-degenerate integrable Hamiltonian system that exhibits instability in the sense that there are orbits with large action variation. This kind of topological instability is sometimes called the \textit{Arnold diffusion}. Indeed, it was conjectured \cite[pp. 176]{a1} that instability is a common phenomenon in the complement of integrable systems.
Aside the several deep contributions towards this conjecture, especially in  recent years (see e.g. \cite{cy}, \cite{du}, \cite{dls}, \cite{kl}, \cite{klm}, \cite{ma}, \cite{ms}, \cite{x}, and references therein), it is still one of the central problems in Hamiltonian dynamics.

Here, we would like to suggest a different  approach related to the instability problem. We propose to focus on the existence and abundance of a dynamical phenomenon, more sophisticated than instability, which  is {\it ``large'' robustly transitive sets}.
Roughly speaking, a set is transitive if it contains a dense orbit inside, and it is robustly transitive if the same holds for all nearby systems (see Definitions  \ref{def large}, \ref{def RT}). 

The present  paper is devoted to studying the non-local robust transitivity (global or non-global) in symplectic and Hamiltonian dynamics with the goal of better understanding the dynamics in  the complement of KAM tori, and with application to the instability problem.

In the non-conservative context, there are many important recent contributions about robust transitivity. Note that a diffeomorphism of a manifold $M$ is transitive if it has a dense 
orbit in the whole manifold. Such a diffeomorphism is called $C^r$ robustly 
transitive if it belongs to the $C^r$ interior of the set of transitive diffeomorphisms.
It has been known since the 1960's that any (transitive) hyperbolic diffeomorphism is $C^1$ robustly transitive. 
The first examples of non-hyperbolic $C^1$ robustly transitive sets are credited to M. Shub \cite{sh} and R. Ma\~n\'e \cite{mn}. For a long time their examples remained unique. Then,  C. Bonatti and L. D\'iaz  \cite{bd} introduced a semi-local source for transitivity, called $blender$, which is  $C^1$ robust.  Using this tool they constructed new examples of robustly transitive sets and diffeomorphisms. For recent results involving blenders, see  \cite{bd08}.
For the recent surveys on this topic and  robust transitivity on compact manifolds, see \cite[chapters 7,8]{bdv}, \cite{ps}, \cite{psh}.

In this paper, we develop the methods of robust transitivity within the context of  symplectic
and Hamiltonian systems. We apply them for the nearly integrable
symplectic and Hamiltonian systems with more than two degrees of freedom. Following this approach, we introduce open sets of such Hamiltonian or symplectic diffeomorphisms exhibiting \textit{large} robustly transitive sets and containing  integrable systems in their closure. 
Then, the instability (Arnold diffusion) is obtained as a
consequence of the existence of large robustly transitive sets.  We want to point out that the results obtained also include systems not necessarily close to integrable ones. 

We also obtain good information about the structure and dynamics of the robustly transitive sets that yields to topological mixing and even ergodicity.
These are the scope of theorems stated in Sections 1.2 - 1.6.

We would like to compare the usual notion of instability (i.e. Arnold diffusion as treated in \cite{dls, cy, x})  with robust transitivity (or topological mixing) obtained in the thesis of our theorems. Observe that the usual notion of instability is  a $C^0$ robust property since it depends only on a finite number of iterations. However, there are no topologically mixing or transitive systems which are  $C^0$ robust (see also Section \ref{s problem ph}).

Let us  emphasize that in this paper we deal with the $C^r$-topology for any $r=1, \dots, \infty$. 
We also work  with  non-compact manifolds. 

Section  \ref{s pre} 
 introduces some definitions and notations; in Sections \ref{s rob} - \ref{s intro dichotomy}  the main theorems are stated.
The two main ingredients used in the proofs are described informally in Sections \ref{s intro blender} and  \ref{s intro ifs}. Finally, Section  \ref{s heuristic} provides a heuristic explanation of how these ingredients are combined and used.

\bigskip

\subsection{Preliminaries and definitions}\label{s pre}

Some of the definitions below are standard  in the literature so we only highlight the ones that are not common.   

Let $M$ be a boundaryless Riemannian manifold  (not necessarily compact) and   $f : M \longrightarrow  M$ be a $C^r$ diffeomorphism of a manifold $M$. From now on we assume that $r\in [1,\infty]$. We denote by $\Diff^r(M)$  the space of $C^r$
diffeomorphisms of $M$ endowed with the uniform $C^r$ topology.

An $f$-invariant subset $\Lambda$ is \textit{partially hyperbolic} if its tangent bundle $T_{\Lambda}M$ splits as a Whitney sum of $Df$-invariant subbundles:   
$$T_{\Lambda} M = E^u\oplus E^c \oplus E^s,$$
and there exists a Riemannian metric on $M$, a positive integer $n_0$ and constants $0< \lambda <1$ and $\mu > 1$ 
such that for every $p\in \Lambda$,
$$ 0<\parallel D_p f^{n_0} |_{E^s} \parallel< \lambda<  m(D_p f^{n_0} |_{E^c} ) \leq
\parallel D_p f^{n_0} |_{E^c} \parallel<\mu< m(D_p f^{n_0} |_{E^u} ).$$

The co-norm $m(A)$ of a linear operator $A$ between Banach spaces is defined by 
$ m(A) :=\inf\{  \parallel A(v) \parallel  ~ : ~ \parallel v\parallel  =1\} $. 
The subbundles $E^u$, $E^c$ and $E^s$ are referred to the unstable, center and
stable bundles of $f$, respectively. 

A partially hyperbolic set is called \textit{hyperbolic} if  its center bundle is trivial, i.e. $E^c = \{0\}$. 

\begin{definition}[domination] 
Let $f$ and $g$ be two diffeomorphisms on manifolds $M$ and $N$ respectively. Suppose that $\Lambda \subset M$ is an invariant hyperbolic set for $f$.  
We say that $g $ is dominated by $f|_{\Lambda}$ if $\Lambda \x N$ is a partially hyperbolic set for $f\x g$, with $E^c=T N$.
\end{definition}

 The {\it homoclinic class of a hyperbolic set}  is the closure of the transversal intersections of its stable and unstable manifolds. In the case of a hyperbolic periodic point $P$ of a diffeomorphism $F$, we denote its homoclinic class by $H(P,F)$. Moreover, for any $G$ nearby $F$, we denote by $P_G$ the {\it analytic continuation} of $P$ and by  $H(P_G,G)$ its homoclinic class.

\begin{definition}[weak hyperbolic point]
Let $p$ be a hyperbolic periodic point of $g$ of period $k$, we say that $p$ is
{\it $\delta$-weak hyperbolic} if 
$$ 1-\delta <  m(D_p g^k |_{E_p^s} ) < \parallel D_p g^k |_{E_p^s} \parallel <1 < m(D_p g^k |_{E_p^u}) < \parallel D_p g^k |_{E_p^u}  \parallel< \frac{1}{1-\delta}.$$
\end{definition} 

\bigskip
Let $X$ be a metric space and $F:X \rightarrow X$ a continuous transformation. A set $Y \subset X$ (not necessarily compact) is
\textit{transitive} for $F$ if for any $U_1, U_2$ open in $X$, such that
$U_i\cap Y\neq \varnothing$, there is some $n$ with $F^n (U_1)\cap U_2 \neq
\varnothing$. If in addition, for any open sets $U_1 , U_2\subset Y$ (in the
restricted topology), there is some $n$ with $F^n (U_1)\cap U_2 \neq
\varnothing$, then we say $Y$ is \textit{strictly transitive}.   
A stronger property is \textit{topological mixing}, where $F^n (U_1)\cap U_2 \neq \varnothing$ holds for \textit{any} sufficiently large $n$. Similarly we define \textit{strictly topologically mixing}. 

In the next definitions we denote by $\cD^r$ a subset of $\diff^{r}(M)$ with the $C^{r}$ topology. 

\begin{definition}[continuation]\label{def continuation}
A set $X \subset M$ of $f$ has {\it continuation} in $\cD^r$ if there
exist an open neighborhood $\cU$ of $f$ in $\cD^r$ and a continuous map $\Phi:\cU \to \mathcal{P}(M)$ such that $\Phi(f)=X$, where $\mathcal{P}(M)$ is the space of all subsets of $M$ with the Hausdorff
topology.
Then, $\Phi(g)$   is called  the \textit{continuation} of $X$ for $g$.

\end{definition} 

\begin{remark}
Note that it is not assumed that the continuation is neither homeomorphic to  the initial set nor invariant.  Compare with Definition \ref{def strong continuation}. 
\end{remark}

\begin{definition}[exceptional set]\label{def except} 
Let $\Lambda$ be a partially hyperbolic set. We say that $X$ is an {\it exceptional subset} of $\Lambda$ if $X\subset\Lambda$ and for any central leaf $L$ of  $\Lambda$, 
the closure of $X\cap L$ in $L$ has zero Lebesgue measure in $L$.  
\end{definition}

\begin{definition}[large set]\label{def large} 
We say that a set $X$ contained in $\Lambda$ is large inside $\Lambda$  if the Hausdorff distance of $\Lambda$ and the interior of $X$ in $\Lambda$  is small. 
\end{definition}

\begin{definition}[compact robustly transitive set]\label{def RT1}
A compact set $Y \subset M$ is $\cD^r$ \textit{robustly (strictly) transitive} for $f\in\cD^r$, if for any $g \in \cD^r$ sufficiently close to $f$, the continuation  of $Y$ does exist and it is (strictly) transitive for $g$. In the same way one may define robustly (strictly) topologically mixing. 
  \end{definition} 

\begin{definition}[non-compact robustly transitive set]\label{def RT}
If $Y$ is not compact, then $Y$ is called $\cD^r$ \textit{robustly (strictly) 
transitive} if  it is the union of an increasing  sequence of compact  $\cD^r$ robustly
(strictly) transitive sets.
In the same way one may define robustly (strictly) topologically mixing for non-compact sets. 
\end{definition}

A periodic point $p$ of $f$ of period $n$ is called \textit{quasi-elliptic} 
if $D_pf^n$ has a non-real eigenvalue of norm one, and all eigenvalues of norm one are non-real. 
If in addition all eigenvalues have norm one then it is called  {\it elliptic}.

A point  $x$ is {\it non-wandering} for a diffeomorphism $f$ if for any neighborhood $U$ of $x$ there is $n\in \NN$ such that $f^{n}(U)\cap U \neq \vazio$. By $\Omega(f)$ we denote the set of all non-wandering point of $f$. A point $x$ is called (positively) {\it recurrent} for $f$ if
$ \liminf_{n  \rightarrow  +\infty} {\rm dist}(x, f^{n}(x))=0.$
A diffeomorphism is said \textit{recurrent} if Lebesgue almost all points are recurrent.

\noindent {\bf Notations.}  
Throughout this paper sometimes we use the following notation:
\begin{itemize}
\item [-]  $X\subset\subset Y$ means that $\overline{X}$ is a compact subset of the interior of $Y$;  
\item [-]  $\UU^{n}:=(0,1)^{n}$ and $\DD^n$ is the unit open ball in $\RR^n$.
\item [-]  $\TT^n$ is the $n$-dimensional torus (with the standard metric).
\end{itemize}

Let us recall some basic facts and definitions of symplectic topology.
A symplectic manifold is a $C^{\infty}$ smooth boundaryless manifold $M$
together with a closed non-degenerate differential 2-form $\omega$. We denote it
by $(M,\omega)$, but sometimes we just write $M$.
Examples of symplectic manifolds are orientable surfaces, even dimensional tori and cylinders, and the cotangent bundle $T^*N$ of an arbitrary smooth manifold. 
A $C^1$ diffeomorphism $f$ is \textit{symplectic} if $f$ preserve $\omega$; i.e.
$f^{*}\omega=\omega$. We denote by $\Diff^r_{\omega}(M)$  the space of $C^r$
symplectic diffeomorphisms of $M$ with the $C^r$ topology, $1\leq r\leq \infty$.
If the symplectic form $\w$ is exact, that is, $\w=d\alpha$ for some 1-form
$\alpha$, and $f^{*}\alpha - \alpha=dS$ for some smooth function $S:M\to \RR$,
then we say that $f$ is an exact symplectic diffeomorphism.

Let $(M, \w)$ be a symplectic manifold and let
$H:\RR\x M \to \RR$ be a $C^r$ function  called the (time dependent) Hamiltonian. For any $t\in \RR$,  the vector field $X_{H_{t}}$ determined by the condition 
$$\w (X_{H_t}, Y)=dH_t(Y) {\rm ~ or ~ equivalently~} i_{X_{H_t}}\w=dH_t$$
is called the {\it Hamiltonian vector field} associated with $H_{t}:=H(t, \cdot)$ or the symplectic gradient of $H_{t}$. 
The Hamiltonian $H$ is called time periodic if $H_{t}=H_{t+T}$ for some $T>0$.
A diffeomorphism is called {\it Hamiltonian diffeomorphism} if it is the time-one map of some time periodic Hamiltonian flow. 

An orbit is called quasi periodic if its closure, say $\cT$, is diffeomorphic to a torus and the dynamics on $\cT$ is conjugate to an irrational rotation on the torus. 

A Hamiltonian on a $2n$-dimensional manifold is called {\it completely integrable} if it has $n$ independent integrals in involution. Recall that an integral is a smooth real function on $N$ (or $N\x \RR$ in the case of time dependent Hamiltonian) which is constant along the orbits of the Hamiltonian flow. 

A Hamiltonian is called {\it integrable} if it is locally completely integrable. A diffeomorphism is called integrable if it is the time-$T$ map (for $T>0$) of some integrable Hamiltonian flow. See  Section \ref{s ifs rec} for more details and precise definitions.  

Throughout this paper all diffeomorphisms and Hamiltonians on non compact manifolds are assumed to have finite norm in their corresponding $C^r$ topology.

\subsection{Robustly mixing sets for symplectic diffeomorphisms}\label{s rob}
We now state  our main result concerning robust mixing  in the context of symplectic diffeomorphisms. Roughly speaking, next theorem states that if the product of a hyperbolic basic set
 with  any non-wandering dynamics is partially hyperbolic, then we can perturb
the initial system  in such a way that a {\it large} robustly topological mixing set is obtained.

\begin{maintheorem} \label{thm A}
Let $M$ and $N$ be two symplectic manifolds so that $N$ is compact, and $1\leq r\leq \infty$. Let $F=f_1\x f_2$  where $f_1\in  \diff^r_\w(M)$, $f_2\in  \diff^r_\w(N)$. Let $V\subset M$ be open and $\Gamma=\Lambda\x N$ such that

{\rm(a)}  $\Lambda=\bigcap_{n\in \ZZ} f^n_1(V)$ is a topologically mixing non-trivial hyperbolic set, 

{\rm(b)} $f_2$ is
dominated by $f_1\vert_{\Lambda}$, 

{\rm(c)}   $f_2$ has a $\delta$-weak hyperbolic periodic point, and $\delta>0$ is small enough. 

\noindent  Then, 
there exists a connected open set  $~\cU\subset  \diff^r_\w(M\x N)$ and a periodic point $P\in \Lambda\x N$ of $F$ such that 
\begin{enumerate}
\item $F$  is contained in the $C^r$ closure of $\cU$.
\item  For any $G\in\cU$, the set $\Gamma_G:=H(P_G;G)\bigcap_{n\in\ZZ} G^n(V\x N)$  verifies
\begin{itemize}
\item [(2.1)] $\Gamma_G$ is a (partially hyperbolic) robustly topologically  mixing set. 
\item [(2.2)]  $\Gamma_G$ tends to $\Gamma$ in the Hausdorff topology as G tends to F.
\end{itemize}
\end{enumerate}

\end{maintheorem}

 Roughly speaking, the item (2.2) shows that the robustly transitive set $\Gamma_G$ is large along $N$, since it is close in the Hausdorff topology to $\Lambda\x N$. 
Observe that for the initial system $F$, the homoclinic class of $P$ could be small along $N$. In fact, it could hold that $\Lambda \x N$ is non-transitive and  $H(P,F) \cap (\Lambda \x N)= \Lambda \x \{q\}$, for some $q\in N$.

Related to the problem of instability, note  that for the symplectic maps in the thesis of Theorem \ref{thm A}, trajectories do not only ``drift'' along the $N$ but belong to a large transitive set, and one obtains robustly transitive sets with ``large variation of action" (see also Theorems \ref{thm C} and \ref{thm D}).

The proof of Theorem \ref{thm A} follows essentially from  Theorem \ref{thm A'} which could be considered as a parametric version of Theorem \ref{thm A}. This parametric version gives a better description of the open set $\cU\subset\diff^r_\w(M\x N)$ involved in the thesis of Theorem \ref{thm A}.
Moreover, Theorem \ref{thm A'} considers the case where the manifold $N$ is not compact. 

\begin{remark}\label{rem A generic} From a set of results proved in the $C^1$ topology (\cite{bdp, dpu, ht, sx}) and  the theorem  by Zehnder and Newhouse (about transversal homoclinic points in the presence of non hyperbolic periodic points, cf.  Theorem \ref{thm ZN}), we conjecture that the hypothesis of Theorem \ref{thm A} are optimal. For more details see Section \ref{s hypo thm A}.
\end{remark}

\begin{definition}
A diffeomorphism satisfies the {\textsf{H.S.}} condition if it has a nontrivial transitive hyperbolic invariant set.  
A diffeomorphism $f$  satisfies the {\textsf{A.H.}} condition
 if any neighborhood of $f$  (in the corresponding space with the $C^r$ topology) contains a diffeomorphism satisfying {\textsf{H.S.}} condition. 
In other words, the $C^r$ closure of the set of diffeomorphisms satisfying the {\textsf{H.S.}} condition coincides with the set of diffeomorphisms satisfying the {\textsf{A.H.}} condition. 
 Similar notions will be used for the Hamiltonian systems.
\end{definition}

If the diffeomorphism $f_2$ in Theorem \ref{thm A} is integrable and $N$ is compact then we can prove that the (strong) continuation of $\Gamma$ (which is a large subset) is in fact robustly transitive. So, we have the following. 

\begin{maintheorem} \label{thm B}
Let $r\geq 1,\;$ $f_1\in  \diff^r_\w(M)$ satisfying the {\textsf{H.S.}} condition  and $f_2\in \diff^r_\w(N)$ be an  integrable
diffeomorphism on a compact boundaryless symplectic manifold  $N$.   
Then $f_1 \x f_2$ is in the $C^\infty$ closure of an open set $\cU\in  \diff^r_\w(M\x N)$ such that any $F\in \cU$ has a robustly  (strictly) transitive set  whose projection on $N$ is onto. 
\end{maintheorem}

Observe that in the thesis of Theorem \ref{thm A} we do not state that the robustly transitive set projects onto $N;$ actually, exceptional set outside the homoclinic class could appear. This is not the case in Theorem \ref{thm B}. This stronger conclusion follows from the fact  that the diffeomorphism $f_2$ in the hypothesis of Theorem \ref{thm B} is integrable, which is not the case for Theorem \ref{thm A}.  To highlight the differences, compare 
Theorem \ref{thm rec} and Theorem \ref{thm ifs min}. On the other hand, observe that in Theorem \ref{thm B} the hypotheses on $f_1$ are weaker than the ones in Theorem \ref{thm A}.

\begin{remark}\label{rem A.H.} To highlight the broad range of application of Theorem \ref{thm B}
observe that the following systems satisfy the {\textsf{H.S.}} condition:
\begin{itemize}
\item [-] any $C^r$ generic perturbation of a symplectic diffeomorphism with an elliptic or quasi-elliptic periodic point,
\item   [-] any $C^r$ generic perturbation of an integrable system (i.e. generic nearly integrable diffeomorphism),
\item  [-]  any diffeomorphism in a $C^1$ open dense subset of $\diff^1_\w(M)$,
\item  [-] any $C^{r}$ surfaces diffeomorphism with positive entropy, if $r>1$. 
\end{itemize}
These are consequences of the results in \cite{ne} (cf. Theorem \ref{thm ZN}) and \cite{kat}. In fact, it is expected that the {\textsf{A.H.}} condition holds for any diffeomorphism in $\diff^r_\w(M)$. See also Section \ref{s hypo thm A}.
\end{remark}

We would like to observe that in  all our  main results the robustly transitive sets are in fact robustly topological mixing for some power of the diffeomorphism.

\subsection{Robust transitivity and instability of nearly integrable Hamiltonian systems}\label{s robHam}

Theorems \ref{thm A} and \ref{thm B} could be stated in the context of  exact symplectic and Hamiltonian
diffeomorphisms, and also time-dependent Hamiltonians. 
The following theorem concerns with the 
class of  Hamiltonian  systems that contains the so-called {\it a priori} unstable integrable Hamiltonian systems.
This theorem can be considered as a Hamiltonian version of Theorem \ref{thm B}.  

\begin{maintheorem} \label{thm C}
Let  $H_{0}(p, q, x,y, t) = h_{1} (x,y,t) +  h_{2} (p)$ be a time-periodic Hamiltonian,
where $t\in \TT:= \RR/\ZZ$ is the time, $(p,q) \in \aA^{2m}=\DD^{m}\x\TT^{m}, ~ and ~ (x,y )\in  \aA^{2n}=\DD^{n} \x \TT^{n}$.
Suppose that $h_{1}$ satisfies the {\textsf{A.H.}} condition and $h_{2}(p)=p\cdot {}^\mathrm{t}p$ (where $ {}^\mathrm{t}p$ means transpose of $p$).
Then for any compact set $K\subset \aA^{2m}$,  there exists a  $C^{r}$ ($r= 2, \dots, \infty$) open set $\cH$ of time periodic Hamiltonians  that contains $H_0$ in its $C^{\infty}$ closure and any Hamiltonian  $H\in \cH$ exhibits a robustly  transitive invariant set  such that its  projection on $\aA^{2m}$ contains  $K$.
\end{maintheorem}

This theorem is related to the results in \cite{dls}, \cite{cy} and \cite{x}, where the case $m=n=1$ and $h_{1}$ integrable (with a saddle loop)  were considered. By quite different methods (and under extra conditions)  it was shown that generic time-periodic perturbations of $H_0$ exhibit instability along the variable $p$ (i.e. along the interval $\DD^{m}, (m=1)$). 

Here, $H_0$ is relaxed from such conditions and it is shown that $H_0$ is in the closure of some open set $\cH$ of Hamiltonians, in which all systems exhibit robustly transitive (or topological mixing) sets which contains arbitrary large regions of $\aA^{2m}$. In particular, a very strong form of instability  is present for all systems in $\cH$.

\begin{maintheorem} \label{thm D}
Let  $h_1$ and $h_2$ be two time periodic Hamiltonians acting on manifolds $M$ and $N$ respectively, both with finite volume. Assume that $h_1$ satisfies
 the  {\textsf{H.S.}} condition and $h_2$ has an elliptic periodic point. Let $H_{\veps} = h_{1}  + \veps h_2$. 
 
 Then, for a sufficiently small $\veps>0$ and  any $\alpha>0$,   there exists a  $C^{r}$ open set $\cV$ of time periodic Hamiltonians on $M\x N$  that contains $H_\veps$ in its $C^{\infty}$ closure and any Hamiltonian  $H\in \cV$ exhibits an invariant set $\Upsilon$ such that
\begin{enumerate}
\item  the projection of $\Upsilon$  on  $N$ has volume $> {\rm vol}(N)-\alpha$,
\item $\Upsilon$ is robustly transitive.
\end{enumerate}
\end{maintheorem}

The type of systems that appear in Theorem \ref{thm D} resembles the sometimes called ``slow-fast  systems'' (see e.g. \cite{akn}). 

\subsection{Ergodic measures with large support}\label{s ergodic zero}
The existence of  ergodic measures with large supports is a very interesting phenomenon, specially for nearly integrable systems and for compact sets in the complement of KAM tori (recall that KAM tori are ergodic components). In this direction we have the following result.

\begin{maincorollary}\label{cor ergodic}
The robustly transitive (or topologically mixing) sets obtained in Theorems \ref{thm A}, \ref{thm B}, \ref{thm C}  and \ref{thm D}
are contained in the support of an ergodic measure.
\end{maincorollary}

This is a direct consequence of our main theorems together with the results in  \cite[Theorem 3.1]{abc}, where a nice property of homoclinic classes was discovered:  {\it any homoclinic class is the support of an ergodic measure}. 
See also Section \ref{s ergodic} for more discussions about ergodicity.

\subsection{Hausdorff dimension of the transitive sets and ergodic measures}\label{s haus}
It is clear from the statement of Theorem \ref{thm A} (as well as from the other main theorems)  that $\dim_H(\Gamma)=\dim_H(\Lambda)+2n$, where $\dim_H$ is the Hausdorff dimension. As a matter of fact, the Hausdorff dimension of certain hyperbolic sets (e.g. some hyperbolic sets in the vicinity of elliptic points) 
may be arbitrarily close to the dimension of the ambient manifold. So, one may find examples for which  $\dim_H(\Gamma)$ is close to $2n+2m$. 
Consequently, in the context of Corollary \ref{cor ergodic}, there exist  ergodic measures with large (non-local) supports and with Hausdorff dimension arbitrarily close to the (full) dimension of the ambient manifold.

\subsection{A dichotomy: transitivity vs.  escaping to infinity}\label{s intro dichotomy}
Observe that once we remove the hypothesis on the compactness of $N$ (and hence let it  have infinite volume) in Theorem \ref{thm A} then a diffeomorphism $G$ near  $F$ may have wandering points in the continuation of $\Gamma$. However, the proof of Theorem \ref{thm A} leads to the following dichotomy: {\it  there exist either large robustly transitive sets or wandering orbits converging to infinity.}
See Section \ref{sec inst vs. rec} for the proof and more details.\\

In the next two subsections we explain the main tools involved in the proofs of the previous theorems. These tools, namely symplectic blenders and iterated functions systems,  could be interesting in themselves.

\subsection{Symplectic blender} \label{s intro blender}
A fundamental ingredient in the proofs is a new tool in symplectic dynamics called \textit{symplectic blender}, a semi-local source of robust transitivity. It is based on the seminal work of Bonatti and D\'iaz \cite{bd}.  
The symplectic blender provides \textit{robustness} of the density of the stable and unstable manifolds of a hyperbolic periodic point, in any compact region, which implies robustness of transitivity or even topological mixing. Theorem \ref{thm blender} is the main result in this direction.
 We believe that the construction of symplectic blenders presented in this paper could have more applications in symplectic dynamics. For more details, see the introduction of Section \ref{sec blender}.

\subsection{Iterated function systems}\label{s intro ifs}
By an iterated function system we mean the action of a semi-group generated by a (finite) family of diffeomorphisms. 
Another main ingredient in the proofs is that we reduce the problem of transitivity to  one about iterated
function systems. That is, the problems related to instabilities and (robust) transitivity can be formulated for iterated function systems. 
This approach is very convenient to deal with the \textit{whole} structure of homoclinic intersections associated to a normally hyperbolic submanifold. Moreover, iterated function systems are used to obtain symplectic blenders.  Consequently, part of this work is devoted to the study of certain iterated function systems.   The main novelties
in this part - beyond the above mentioned reduction - are the results on global and local dynamics of iterated function systems (cf. Theorems \ref{thm rec} and \ref{thm ifs min}). 
We believe this type of problems are in itself worth of study.

\subsection{Heuristic model}\label{s heuristic} As a basic model to understand the strategies of the proofs of Theorems \ref{thm A} and \ref{thm B}, one may consider perturbations of the product of a horseshoe and an integrable twist map, where applying  results on the \textit{iterated function systems} of recurrent diffeomorphisms lead to the \textit{minimality} of (strong) stable and unstable foliations (giving instability as well). Then, using  the \textit{symplectic blender} one can show that transitivity (or even topological mixing) appears  in a \textit{robust} fashion. 
Observe that the hypotheses of Theorem \ref{thm B} are in the  framework of the described strategy. Similar arguments are applicable in the more general context of Theorem \ref{thm A} that includes non-integrable systems.

Note specially that we do not use any KAM-type invariant sets in the proof.
Instead, \textit{recurrence} has an important role. Therefore, some difficulties such as appearance of large gaps between Diophantine tori (known as {\it the large gap problem}) are not present here. \bigskip


The paper is organized as  follows.
In Section \ref{sec IFS}, we study transitivity of two different kinds of
iterated function systems (IFS). Namely, the IFSs of expanding maps, and the
IFSs of recurrent diffeomorphisms.  In certain cases, the minimality of the IFSs is obtained, which is essential in the proof of Theorem \ref{thm B}. In Section \ref{sec blender},  we introduce the
symplectic blenders and we use the expanding IFSs to build them.
In Section \ref{sec proof A} we prove a parametric version of Theorem  
\ref{thm A}. We use this theorem and its proof in Section \ref{sec instability} where we prove our main theorems and also the dichotomy stated in  Section \ref{s intro dichotomy}.
Finally, in Section \ref{sec remarks} 
several remarks
and open problems related to the main results are discussed.
In particular, in Section \ref{applications}, we discuss the possibility of extending the methods and results introduced here to many fruitful contexts, including non-artificial mechanical and geometric problems.\\

\textbf{Acknowledgment.}
The authors would like to thank Lorenzo D{\'\i}az, Vadim Kaloshin, Patrice Le Calvez, Rafael de la Llave, Mike Shub and Marcelo Viana for helpful conversations and comments.
M.N. would like to express his deepest thanks to Jacob Palis for enormous encouragement, and to Marcelo Viana for constant support.
This work was done while M.N. was at IMPA, and some parts were written while he was at IASBS in Zanjan.  M.N. was partially supported by IMPA and CNPq.

\section{Iterated function system} \label{sec IFS}

In this section we study the transitivity of some iterated function systems (IFS). 
Loosely speaking, in the IFS, instead of taking iterations by only \textit{one map}, all the possible compositions and iterations of \textit{several maps} are considered. As a consequence, a point $x$ may have an infinite number of orbit branches.
The transitivity of an iterated function system of expanding maps has a
fundamental role in the construction and 
properties of blenders (see Section \ref{sec blender}). 
And the  transitivity of an iterated
function system of  symplectic maps shall be used in the proof of the density of
the (strong) stable and unstable manifolds of partially hyperbolic skew products (see Section \ref{sec proof A}).

Let $g_1, g_2, \dots , g_n$ be  maps defined on a metric space $X$.
The  iterated function system  $\cG(g_1, g_2,$ $\dots,$ $g_n)$ is the action of the semi-group generated by $\{g_1, g_2,$ $\dots,$ $g_n\}$ on X.
We use the notion of multi-index $\sigma=(\sigma_{1}, \dots, \sigma_{k})\in \{1,2, \dots,n\}^{k}$ for $g_{\sigma}= g_{\sigma_{k}}\circ \cdots \circ g_{\sigma_{1}}$. We also denote the length of the multi-index $\sigma$ by $| \sigma |= k$. 

An orbit of $x\in X$ under the iterated function system  $\cG=\mathcal{G}(g_1, g_2, \dots, g_n)$ is a sequence $\{ g_{_{\Sigma_k}}(x)\}_{k=1}^{\infty}$ where $\Sigma_{k}=(\sigma_{1}, \dots, \sigma_{k})$ and $\{\sigma_{i}\}_{i=1}^{\infty} \in \{1,2, \dots,n\}^{\mathbb{N}} $. 

The $\cG$-orbit of  $x$ denoted by $\orbit_{\cG}^+(x)$ is the set of points lying on some orbit of $x\in X$ under the IFS  $\cG$. The $\cG$-orbit of  a subset  $U\subset X$  is defined as the union of all its orbits, i.e. $\orbit_{\cG}^+(U)= \bigcup_{x \in U} \orbit_{\cG}^+(x)$.

Similarly, we denote  $\orbit_{\cG}^-(x)$ as the set of points such that $x$ lies on (some of) their orbits. 

Results similar to the one proved here can be found in \cite{la}, \cite{mo}, \cite{thesis}, \cite{GHS} and \cite{kn}.

\begin{definition}
The IFS $\mathcal{G}(g_1, g_2, \dots, g_n)$ is said \textit{transitive} if the $\cG$-orbit of any open set is dense. A set $U$ is called transitive for $\cG$ if the $\cG$-orbit of any open subset of $U$ is dense in $U$. This is equivalent to the existence of some point with a dense $\cG$-orbit in $U$.
\end{definition}

\begin{remark}
In similar way one defines the IFS of maps $g_{i}:U_{i}\subset X \to X$. In this case, the possible compositions of $g_{i}$'s depend on each point:
$g_{i}(U_{i})$ is not necessarily a subset of $U_{j}$ and so
$g_{j}\circ g_{i}$ is only defined on $U_{i}\cap g^{-1}_{i}(U_{j})$.
\end{remark}

\subsection{IFS of contracting and expanding maps}\label{s ifs contract}
 In this section we study the transitivity of the iterated function systems of contracting and expanding maps. In Section \ref{s ifs contracting} we give the results for  contracting ones and the theorem for expanding are sketched in Section \ref{s ifs expanding} (the details are left to the reader). 

\subsubsection{IFS of contracting  maps}\label{s ifs contracting}

A map $\phi$ on a metric space $(X,d)$ is called {\it contracting} if   there is a constant 
$0<K<1$ such that $d (\phi(x), \phi(y))< K d(x,y)$, 
for all $x, y \in X$. 
The \textit{contraction bound}, if it exists, is a number $\lambda>0$ for which $\phi$  satisfies $\lambda d(x,y) < d (\phi(x), \phi(y))$ 
for all $x, y \in X$. 
This constant does not necessarily exist for any contracting map. 
For a generic smooth contracting map $\phi$ on $\RR^n$, the contraction bound does exist if we consider its restriction on a compact domain $U$. 
In this case, the constant is equal to $\inf\{m(D\phi_z) ~:~ z\in U \}$.

\begin{proposition} \label{pro ifs 1}
Let $U \subset \RR^n$ be an open disk containing $0$ and  $\phi: U \to U$ be a contracting map with the contraction bound $\lambda$ and $\phi(0)=0$. Then there exists $k \in \NN $ such that for any  $\veps > 0$ small there exist vectors $c_1, \dots , c_k \in B_{\veps}(0)$ and a number $\delta>0$ such that
$$B_{\delta}(0)\subset \overline{\orbit_{\cG}^+(0)},$$
where $\cG=\cG(\phi, \phi+c_1, \dots , \phi+c_k )$. 

Moreover, these properties are robust in the following sense:\\
Let $\phi_{0}=\phi$ and $\phi_{i}=\phi+c_{i}$. Let $\cU_{i}$ be a set of contracting maps $C^0$ close to $\phi_{i}$ such that their contraction bounds are also close to that of $\phi_i$.
Then the same is true if at each iteration in the $\cG$-orbit of $0$ one replaces the corresponding  $\phi_{i}$ by any $\tilde{\phi}_i\in \cU_{i}$.
\end{proposition}

Observe that if $\phi$ is smooth, all constants depend only on $m(D\phi(0))$. 

In order to prove this proposition, we start with a non-perturbative version of it, which also clarifies the robustness of transitivity.

\begin{definition}\label{prop cont IFS} 
We say that an iterated function system $\cG(\phi_1, \dots , \phi_k )$ of contracting maps has the {\it covering property} if there is an open set $\cD$ such that 
$$\cD \subset \bigcup_{i=1}^k \phi_i(\cD).$$
We say that the set of (unique) fixed points $z_i$'s of $\phi_i$'s is {\it well-distributed} if any open ball of diameter $d$ and centered in $\cD$ contains some $z_i$, where 
$$d\geq\max\{r\mid \forall x\in \cD, \exists i, B_r(x) \subset \phi_i(\cD) \}.$$ 
\end{definition}

\begin{figure}[]
\begin{center}
\includegraphics[width=5cm]{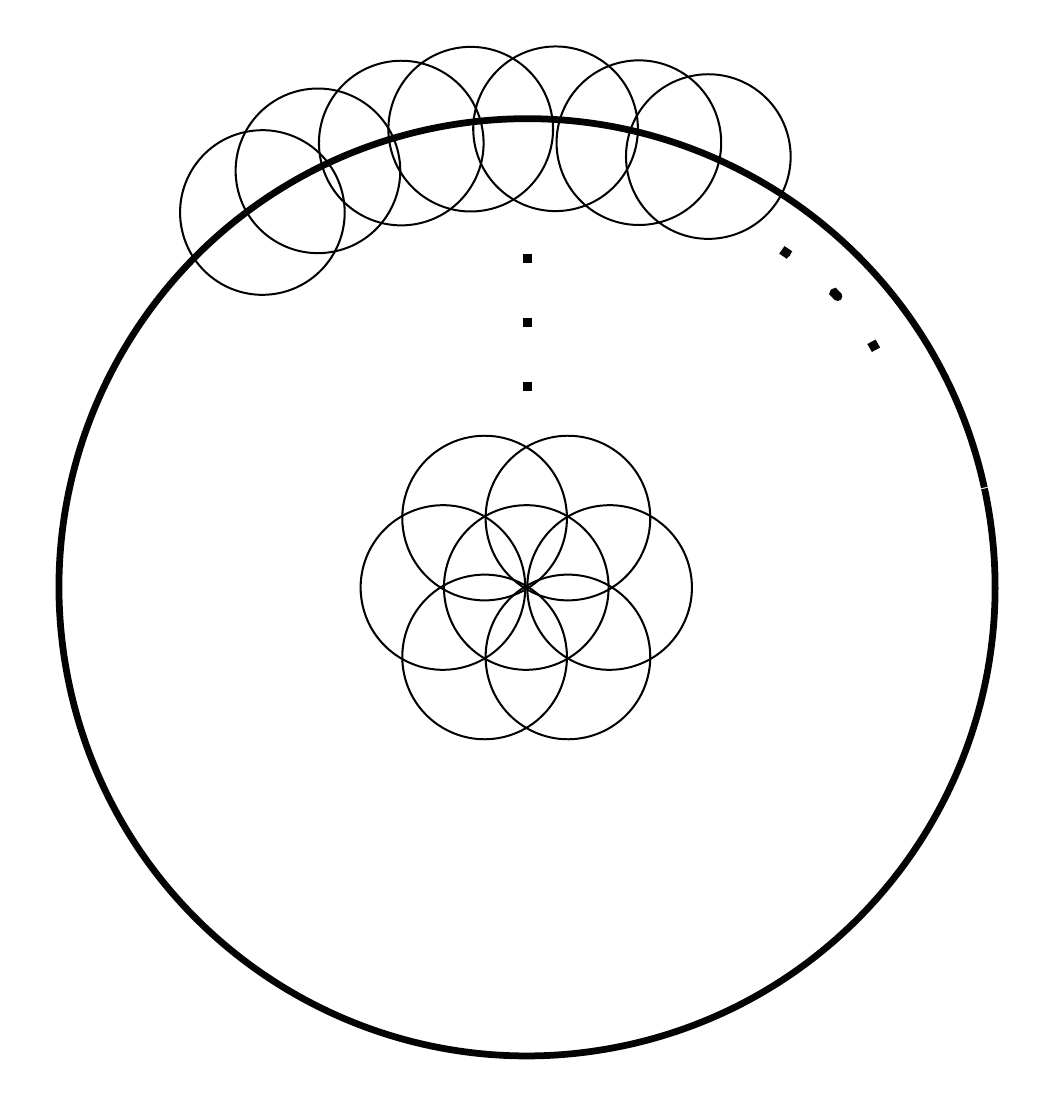} 
\label{fig ifs}
\caption{The covering and well-distributed properties. The disk $\cD$ is the largest one and the other disks are its images under $\phi_i$'s.}
\end{center}
\end{figure}

\begin{proposition} \label{pro ifs 2}
Let $\phi_i: \RR^n \longrightarrow \RR^n$, $i=1, 2, \dots , k$, be contracting maps, and let $z_i$ be their unique fixed points. Suppose that the iterated function system $\cG=\cG(\phi_1, \dots , \phi_k )$ has the covering property on $\cD$. Then for any $x \in \cD$ there exists a sequence $\{\sigma_j\}_{j=1}^{\infty}$ such that for all $j\in \NN$, $\sigma_j \in \{1, 2, \dots, k\}$, and 
$$\phi^{-1}_{\sigma_j} \circ \phi^{-1}_{\sigma_{j-1}} \circ \cdots \circ \phi^{-1}_{\sigma_1}(x) \in \cD.$$
In addition, if the set of fixed points $\{z_i\}_{i=1}^k$ is well-distributed in $\cD$ then 
$$\cD \subset \overline{\orbit_{\cG}^+(0)}.$$ 
\end{proposition}

\Proof
To prove the first part notice that given a point $x \in \cD$, the covering property says that there is $\sigma_1 \in \{1, 2, \dots, k\}$ such that $\phi^{-1}_{\sigma_1}(x) \in \cD$. Then, inductively, one  constructs a sequence  $\{\sigma_j\}_{j=1}^{\infty}$ such that 
$\phi^{-1}_{\sigma_j} \circ \phi^{-1}_{\sigma_{j-1}} \circ \cdots \circ \phi^{-1}_{\sigma_1}(x) \in\cD$.

Now we prove the second part. The well-distributed property yields that for any 
small ball $B_r(x_0)$ in $\cD$, either it belongs to some $\phi_i(\cD)$ or it contains the fixed point of some $\phi_i$. Let $0<K<1$ such that $d (\phi_{i}(x), \phi_{i}(y))< K d(x,y)$ for any $ i \in \{1, 2, \dots, k\}$. If the ball $B_r(x_0)$ is very small then it belongs to the domain of some $\phi_i$, i.e. $B_r(x_0)\subset \phi_i(\cD)$, and so there is $x_1\in\cD$ such that 
 $B_{K^{-1}\cdot r}(x_1) \subset \phi_i^{-1}(B_r(x_0)) \subset \cD$. We may continue this process inductively.  Since the ratio of the balls is increasing exponentially at rate $K^{-1}>1$, after some iteration, it would be large enough to contain the fixed point of some $\phi_i$. This completes the proof. 
\Endproof

\begin{remark}
The well-distributed property yields that for any 
small ball $B_r(x_0)$ in $\cD$, either it belongs to some $\phi_i(\cD)$ or it contains the fixed point of some $\phi_i$.  
\end{remark}

\Proofof{Proposition \ref{pro ifs 1}}
It is enough to show that there exists a number $k$, and certain (small) translations of the map $\phi$, such that the covering property and the well-distributed hypothesis holds in some open ball $B_{\veps}(0)$. Then using Proposition \ref{pro ifs 2} we obtain the density of the $\cG$-orbit of $0$. 
To show that, we consider a cover of the unit ball by $k$ balls of radius $\lambda$. There is a number $C(n)$ such that for any $r<1$, there is a cover of the unit ball in $\RR^{n}$ by $C(n) r^{-n}$ number of  balls of  radius $r$.
So there is a $k_{1}=C(n)2^{n}\cdot \lambda^{-n}$ number of points $c_{i}\in\DD^{n}$ (where $\DD^n=B_1(0)$), such that 
$$\overline{\DD}^{n}\subset \bigcup_i B_{\lambda/2}(c_{i}), $$ 
and so 
$$\overline{\DD}^{n}\subset \bigcup_i ( \phi(\DD^{n})+ c_{i}). $$
Moreover, any ball of radius smaller than $\lambda/2$ is contained in some $\phi(\DD^{n})+ c_{i}$. 

We also choose  a set of $z_{i}$, $i=1, \dots, k_{1}$ which is $\lambda$-dense in the unit ball. So,  any ball of radius greater than $\lambda/2$ in the unit ball contains some $z_{i}$.
Let $c_{i+k_{1}}=({\rm id}-\phi)(z_{i})$. Indeed, $z_{i}$ is the unique fixed point of the contraction $\phi+c_{i+k_{1}}$.

Now, let $k=2k_{1}=C(n)2^{n+1}\cdot \lambda^{-n}$, $\phi_{0}=\phi$ and for $i=1, \dots, k$, $\phi_{i}=\phi+c_{i}$. Again, robustness follows from the fact that the covering property and the well-distributed hypothesis are $C^{0}$ robust properties if the contraction bounds of the nearby maps are close to the initial ones.
\Endproof

\subsubsection{IFS of expanding  maps}\label{s ifs expanding}
A map $\phi$ on a metric space $(X,d)$ is called {\it expanding} if and only if it is invertible and  its inverse is a contracting map.  Observe that in the expanding case, instead of the density of the forward orbit under the IFS, we look for the density of the backward orbits under the IFS.
In particular, the covering property is also formulated for the inverse maps.  Similar propositions to the one for contracting maps can be formulated for expanding maps, and although translation of a map is not a translation of its inverse, the proof of the equivalent propositions for expanding maps are done in the same way.

\subsection{IFS of recurrent diffeomorphisms}\label{s ifs rec}
In this section we study the transitivity of an iterated function system of recurrent diffeomorphisms. 
Here, by recurrent diffeomorphism we mean a diffeomorphism for which almost all points are recurrent. 
Recall that the  Poincar\'e Recurrence Theorem asserts that any diffeomorphism preserving a finite volume form is recurrent.

The results of this section shall be used in the proof of Theorem \ref{thm A} and other main theorems. We first prove several results  for integrable systems and then we generalize them to general recurrent diffeomorphisms. In fact, in Proposition \ref{pro ifs 3} we work  with the IFS of integrable symplectic diffeomorphisms and in Theorem \ref{thm rec} we extend previous proposition to the case of recurrent ones.

Let us first recall the statement of Liouville-Arnold Theorem that allows us to re-define integrable systems as the following.

We say that  $f\in \diff^{r}_{\w}(N)$ is {\it integrable} if there exist open sets $N_{i}$ such that
\begin{itemize}
\item [-] $N=\overline{ \cup N_{i}}$,
\item [-] $N_{i}$'s are mutually disjoint open sets,
\item [-] for any $i$, $N_{i}$ is invariant and diffeomorphic to $\UU^{n}\x\TT^{n}$ by a (symplectic)  diffeomorphism $h_{i}$,
\item [-] any torus $h^{-1}_{i}(\{x\}\x\TT^{n})$ is $f$-invariant and  its dynamics is conjugate to a rotation.
\end{itemize} 

\noindent {\bf Addendum.} We also assume the following assumptions:
\begin{itemize} 
\item [-] bunching condition: 
$$\limsup_{n\to \infty} \frac{1}{n} \log( ||Df^n||/m(Df^n))$$
 is sufficiently close to zero. (Observe that the limit is always zero on $\cup N_{i}$ and if there exists a periodic point in the boundary of any $N_i$, then the larger and smaller eigenvalues are close.)
\item [-] the family $\{N_{i}\}$ is locally finite in $N$, that is, any compact subset of $N$ intersects with only finite number of $\{N_{i}\}$.
\item  [-] the map $\UU^{n}  \ni x\mapsto  \w_x\in\TT^n$  is a local diffeomorphism, 
where $\w_x$ is the rotation vector of $f$ on  $h^{-1}_{i}(\{x\}\x\TT^{n})$. 
\end{itemize}


\begin{lemma} \label{lem tori}
Let $f_1$ be an integrable symplectic diffeomorphism on the symplectic manifold $N$. Then $C^r$-arbitrarily close to $f_1$ there exists another integrable symplectic diffeomorphism $f_2$ 
which is conjugate to $f_1$ by a smooth change of coordinates 
on $N$ such that 
\begin{enumerate}
\item any $f_1$-invariant torus intersects transversally some $f_2$-invariant torus, and vice versa, 
\item given two open sets $U, V\subset N$, there is a chain of tori $\cT_j$, $j=1,2, \dots, s$, invariant for $f_{\sigma_j}$, $\sigma_j=1$ or $2$, such that, each $\cT_j$ (for $j<s$) intersects transversally $\cT_{j+1}$; $\cT_1$ intersects $U$ and $\cT_s$ intersects $V$.
\end{enumerate}
\end{lemma}

\Proof
We construct a symplectic diffeomorphism $\phi\in \diff^{r}_{\w}(N)$ close to the identity such that   $f_2 = \phi \circ f_1\circ \phi^{-1}$ has the desired properties.  

By the assumptions, $N=\cup\overline{N_{i}}$, where $N_{i}$ is diffeomorphic to $\UU^{n}\x\TT^{n}$ by a symplectic diffeomorphism $h_{i}$.  
It is convenient to consider the polar coordinate system on $\UU^{n}\x\TT^{n}$, that is, any point is represented by 
$$(r_{1}, \dots, r_{n}, \theta_{1},\dots, \theta_{n}),$$
where $0< r_{i}<1$ and $\theta_{i} \in \TT^1$.

The construction of $\phi$ has two steps.

{\it Step 1.}  
Let $\psi^{\tau}\in \diff^{r}_{\w}(\UU^{1}\x\TT^{1})$ be the time  $\tau$ map of a completely integrable Hamiltonian flow such that in the polar coordinates we have
\begin{itemize}
\item  [-] $\psi^{\tau}(r,\theta)= (r, \theta)$, for any $\tau  ~{\rm if} ~ r\geq 1$,
\item  [-] $\psi^{1}(\{r=c\})\neq \{r=c\},  ~{\rm if} ~ 0\leq c <1$, 
\end{itemize}
In addition, given an open set $W_{0}$ in the annulus $\UU^{1}\x\TT^{1}$ containing a curve that connect the two boundary components of the annulus we have: 
\begin{itemize}
\item [-]  any two open sets in the unit disk $\{r<1\}$ are connected by a finite chain of circles  $\{r=c_{j}\}$ and $\psi^{1}(\{r=c_{i}\})$ such that the the transition points are contained in $W_0$. 
\end{itemize}
The Hamiltonian flow $\psi^\tau$ can be defined explicitly by 
the Hamiltonian 
$$h_\veps(r,\theta):=\eta(r)\cdot\frac{r^{2}\sin(\theta) + r^{2}\cos(\theta) (1+\veps\cdot \eta(r))^{2}}{1+\veps\cdot \eta(r)},$$
where  $\veps>0$ is small and $\eta$ is a (suitable) non-negative $C^{\infty}$ bump function on the real line such that $\eta(x)>0$ if and only if $x\in (0,1)$.

Now let 
$$\psi_{\tau} =\overbrace{\psi^{\tau}\x \cdots \x \psi^{\tau}}^{n ~ times}.$$ 
Define $\varphi \in\diff^{r}_{\w}(N)$ by
$$\varphi = \left\{ \begin{array}{lcl} 
h^{-1}_{i} \circ \psi_{\tau_{i}} \circ h_{i}  & \mbox{on}  & N_{i} \\
{\rm id}    & \mbox{on}  & N \setminus \cup N_{i}
\end{array}\right.
$$ 
where $\tau_{i}$ is small enough and depends on $N_{i}$.

The smoothness of $\varphi$ on each $N_{i}$ is obvious. The smoothness of $\varphi$ on the boundary of $N_{i}$'s  follows from the following facts: (1) the time $\tau_{i}$ is small enough and depends on $N_{i}$; (2) the sets  $N_{i}$ are mutually disjoint and (3) $\psi$ is equal to the identity on the boundary of $\UU^{n}\x\TT^{n}$.
In order to get $\varphi$ close to the identity, it is enough to take $\tau$ small enough.

The items (1) and (2) of the lemma hold for the pair of diffeomorphisms $f_1$ and $\varphi \circ f_1\circ \varphi^{-1}$, at least on each $N_i$. 
Moreover, given any pair of open sets in $N_{i}$ and any open set $W_{1}$ that intersects all the $f_1$-invariant tori  in $N_i$, there exists a chain as in item (2) of the lemma such that the intersections between the tori are contained in $W_{1}$.

{\it Step 2.}
Let $i>j$ such that  $\partial N_{i} \cap \partial N_{j}$ contains a regular hypersurface (codimension one submanifolds) $S_{ij}$. Then for any such $i,j$ we consider a small open neighborhood $U_{ij}$ of some point of the hypersurface $S_{ij}$. The sets $U_{ij}$ are pairwise disjoint. 
We choose $U_{ij}$ small enough such that $W:=N\setminus \cup U_{ij}$ intersects all the $f_1$-invariant tori.

Let  us define $U^{+}_{ij}=U_{ij}\cap N_{i}$ and $U^{-}_{ij}=U_{ij}\cap N_{j}$.  
Then consider a symplectic diffeomorphism $\varphi_{ij}$ supported in $U_{ij}$ and close to the identity such that 
$$\varphi_{ij}(U^{-}_{ij})\cap U^{+}_{ij}\neq\vazio   \andrm  \varphi_{ij}(U^{+}_{ij})\cap U^{-}_{ij} \neq\vazio.$$

Now we take the composition of the all the above diffeomorphisms to define $\phi\in \diff^{r}_{\w}(N)$, that is 
$$\phi:= (\circ_{ij} \varphi_{ij})\circ \varphi.$$

Now, we define 
$$f_2 := \phi \circ f_1\circ \phi^{-1}.$$
Observe that $\phi=\varphi$ on the open set $W$, and as mentioned before it follows that the items (1) and (2) hold on each $N_i$. So, given any two open sets $U, V$ in $N$, if they intersect the same $N_i$, there is nothing left to prove. Otherwise, there exist   $j_1, j_2$, such that $U\cap N_{j_1}\neq \vazio $ and $U\cap N_{j_2}\neq \vazio$. Then, one may find a chain of $N_i$ from $N_{j_1}$ to $N_{j_2}$ by taking iteration by $f_2$.  
Consequently,  $f_2$ has the desired properties.  
\Endproof

\begin{proposition} \label{pro ifs 3}
Let $T_1$ be an integrable symplectic diffeomorphism on the symplectic manifold $N$.
Then, arbitrarily close to $T_1$ there exists an integrable symplectic diffeomorphism $T_2$ 
on $N$ such that the iterated function system $\mathcal{G}(T_1^{d}, T_2^{d'})$ has a dense orbit, for any $d, d'\in \ZZ$. Moreover, almost all points have dense $\cG$-orbits.
\end{proposition}

\Proof 
Let $T_2$ be  an  integrable diffeomorphism.
Let $\mathcal{S}_0$ be the set of all quasi periodic points of $T_1$ and observe that the set is $T_1$-invariant. Similarly, let $\mathcal{S}'_0$ be the set of all quasi periodic points of $T_2$, which is $T_2$-invariant. The complement of $\mathcal{S}_0\cap\mathcal{S}'_0$ has zero Lebesgue measure.
Let $\mathcal{S}$ be  the set of all points whose orbits under the iterated function system $\mathcal{G}(T_1, T_2)$ belong to $\mathcal{S}_0\cap\mathcal{S}'_0$.

\claim
{\it The complement of $\mathcal{S}$ is contained in a countable union of codimension one submanifolds. So, $\mathcal{S}$ has  total Lebesgue measure.}

\Proofof{Claim}
Let $\cB_{0}:=N\setminus \cS_0$ and $\cB:=N\setminus \cS$. Then, $\cB_{0}$ is contained  in a countable union of codimension one submanifolds. On the other hand, $\cB$ is the orbit of $\cB_{0}$ under the IFS $\mathcal{G}(T_1, T_2)$. So, it is a {\it countable} union of the iterates of $(\cB_{0})$. This proves the claim.

Now we apply Lemma \ref{lem tori} for $T_{1}$ and we obtain $\phi\in \Diff^r_{\omega}(N)$ close to the identity.   Now we set $T_2 := \phi \circ T_1\circ \phi^{-1}$. Define $\cS$ as above and observe that  $\cS$ has total Lebesgue measure.
We want to show that the  orbit of any point of $\cS$ is dense.  

Given two open sets $U, V$, there is a chain of tori $\cT_j$, $j=1,2, \dots, s$, invariant for $T_{\sigma_j}$, $\sigma_j=1$ or $2$, such that each $\cT_j$ intersects (transversally) $\cT_{j+1}$,  $\cT_1$ intersects $U$ and $\cT_s$ intersects $V$. It is not difficult to find an orbit of $\cG$ which shadows this chain. For any $z \in \cS$, there is $n_z$ such that $T_{\sigma_j}^{n_z}(z)$ is close to $\cT_{j+1}$ if $z$ is sufficiently close to $\cT_j$.  The set $\cS$ is $\cG(T_1, T_2)$-invariant. So if  $z \in \cS $ is sufficiently close to $\cT_1$, then  it has a $\cG$-orbit shadowing all $\cT_j$, and therefore there is an orbit from $U$ to $V$.
Moreover, given any point $x\in \cS$ and any open set $U$, there is a finite sequence of tori $\cT_i$, $i=1, \dots, n$, invariant for $T_1$ or $T_2$ (alternatively), such that $x \in \cT_1$, $\cT_n \cap U \neq \vazio$, and for any $i$, $\cT_i$ intersects transversally $\cT_{i+1}$. Then it follows that there exists $\Sigma=(\sigma_1, \dots, \sigma_m)$ such that $T_{_{\Sigma}}(x)\in U$. This completes the proof.
\Endproof

\begin{remark} 
If the set of quasi periodic points is residual then following the same argument in the proof of Proposition \ref{pro ifs 3}, we conclude that the set of all points with dense orbit for $\mathcal{G}(T_1, T_2)$ is also residual.
\end{remark}

The following example helps to understand the above described techniques. 
\example{
Let $N=\mathbb{R}\times (\mathbb{R}/2\pi \mathbb{Z})$ and 
$$T_1 : (I,\theta) \longmapsto (I,~ \theta + h(I) ).$$
In this case, we choose the change of coordinates 
$$\phi : (I,\theta) \longmapsto (I+ \epsilon\cos \theta,~ \theta ).$$
Then we define $T_2 = \phi \circ T_1 \circ \phi^{-1}$. }

\begin{theorem}[Minimal IFS] \label{thm ifs min}
Let $T_1$ be an integrable diffeomorphism on the  symplectic manifold $N$. Let $m=\dim(N)+2$.
Then, arbitrarily close to $T_1$ there exist integrable symplectic diffeomorphisms $T_2, \dots , T_{m}$ on $N$ such that the iterated function system $\mathcal{G}(T_1, \dots , T_{m})$ is minimal (i.e. every point has a dense orbit). \end{theorem}

\Proof Let $T_2$ be the integrable diffeomorphism given by Proposition \ref{pro ifs 3}.
It follows from the claim in the proof of Proposition \ref{pro ifs 3}  that 
 the  complement of $\cS$ is contained in a countable union of codimension one submanifolds of $N$. Let $\cL_1:=\{L_i\}_i$ be the family of those submanifolds and so $\cB:=N\setminus \cS=\bigcup \cL_1$. (Here we use the union $\bigcup$ as the unitary operator, as well.)

Since any point in $\cS$ has a dense orbit, to prove minimality, it is enough to show that there exist integrable diffeomorphisms $T_3, \dots, T_m$ close to $T_{1}$, such that any point  $x\in N\setminus \cS$ has an iterate in $\cS$ for the iterated function system $\cG(T_3, \dots, T_m)$.

We give a proof of this statement using transversality theory. 
It is known that the set of diffeomorphisms transversal to a submanifold $L$ is $C^r$ residual.
On the other hand, if $L$ and $L'$ are two submanifolds of codimensions $m$ and $m'$, respectively, then the transversal intersection of them has codimension $m+m'$.

Therefore, for any $i,j$, there is a residual set $R_{1,i,j} \subset \diff_\w^r(N)$ such that 
for any $f\in R_{1,i,j}$ and $L_i, L_j\in \cL_1$, the submanifolds $L_i$ and $f(L_j)$ are transversal and so the intersection $L_i\cap f(L_j)$ is a countable union of codimension two submanifolds. 
Let $\phi_3\in \cR_1:=\bigcap_{i,j} R_{1,i,j}$ sufficiently close to the identity,  then $\bigcup_{i,j} (L_i\cap \phi_3(L_j))$ is a countable union of submanifolds $L'_{j}$,  $j\in \NN$, of codimension two. 
Let $\cL_2:=\{L'_i\}_i$ be the family of those submanifolds.
We have $\cB \cap \phi_3 (\cB) = \bigcup \cL_2$.

We repeat this argument for any pair of $L_i\in \cL_1$ and $L'_j\in \cL_2$ to get the residual set of diffeomorphisms $R_{2,i,j}$. By choosing  $\phi_4\in \cR_2:=\bigcap_{i,j} R_{2,i,j}$ sufficiently close to the identity,  then $\bigcup_{i,j} (L_i\cap \phi_4(L_j))$ is a countable union of submanifolds $L''_{j}$,  $j\in \NN$, of codimension 3. We define $\cL_3$ of such codimension 3 submanifolds. We have $$\cB \cap \phi_3 (\cB) \cap \phi_4 (\cB \cap \phi_3 (\cB)) = \cB \cap \phi_3 (\cB) \cap \phi_4 (\bigcup\cL_2) = \bigcup \cL_3.$$

Repeating this argument inductively, we get $\phi_3, \dots, \phi_{m-1}$
and the countable set $\cL_{m-2}$,  (i.e., a countable set of codimension $\dim(N)$ submanifolds), such that $\cB \bigcap_{j=3}^{m-1} \phi_j( \bigcup \cL_{j-2}) =\bigcup \cL_{m-2}$.

Now, for a generic diffeomorphism $\phi_m$ which we assume to be sufficiently close to the identity, we have $\cB\cap \phi_m(\bigcup \cL_{m-2})=\vazio$. 
This means that for any $x\in\cB$ there exist $j_1, \dots , j_i$ such that  $\phi_{j_1} \circ \cdots \circ \phi_{j_i}(x)\in \cS$.
Now, we define $T_{j}:=\phi_{i} \circ T_{1}\circ \phi_{i}^{-1}$. It follows that any point  $x\in N\setminus \cS$ has an iterate in $\cS$ for the iterated function system $\cG(T_3, \dots, T_m)$.
Clearly, all $T_{j}$ are integrable (Hamiltonian) diffeomorphisms and close to  $T_{1}$. This completes the proof.
\Endproof

\begin{remark}\label{rem ifs min 3} To apply  Theorem   \ref{thm ifs min} to the proof of the main theorems, it is essential that the number $m$ of generators is finite and independent of $T_1$. However, one may ask about the optimal number $m$ of generators needed to obtain minimality. 
 In fact, one may prove that three generators are enough. To see this, just observe that given $k\in \NN$ and a codimension one submanifold $L$ in $N$, and for a $C^r$ generic diffeomorphism $f\in\diff_w^r(N)$, $f^k$ is transversal to $L$, for any $k\in \NN$. Then, one may easily prove that for a $C^r$ generic $f$, the diffeomorphisms $T_i:=f^{i-2},  \; i=3, \dots, m$, satisfy the transversality conditions in the proof of Theorem \ref{thm ifs min}. We choose $T_{3}=f$ to be close to $T_{1}$. Notice that in this argument we do not assume $T_{3}$ being integrable. \end{remark}
 

Now, we establish a result about transitivity of the IFS of recurrent diffeomorphisms.

\begin{theorem}\label{thm rec}  
Let $T\in \Diff^r_\omega(N)$ be a recurrent diffeomorphism. Then for every $\epsilon>0$,
\begin{enumerate}
\item  there exist $T_1, T_2 \in B_\epsilon(T)\subset \Diff^r_\omega(N)$ such that $\cG(T, T_1, T_2)$ is transitive;

\item  for any open ball $V\subset N$ and any bounded domain $N_c \subset N$, there exist $k\in \NN$ and  $T_1, T_2, \dots, T_k \in B_\epsilon(T)\subset \Diff^r_\omega(N)$ such that $N_c \subset \orbit^-_{\cG}(V)$, where $\cG= \cG(T, T_1, T_2, \dots, T_k)$.
\end{enumerate}

\end{theorem}

\Proof
If $T={\rm id}$, then we choose $\phi_1$ to be an integrable symplectic diffeomorphism on the manifold $N$ such that almost all points are quasi periodic, and $d_{C^r}(\phi_1, {\rm id})<\frac{1}{2}\epsilon$. 
Proposition \ref{pro ifs 3} implies that for any open set $V$ there exists $\phi_2$ in 
$\Diff^r_\omega(N)$ and $\epsilon$-close to the identity in the $C^r$ topology, such that, 
$\orbit^-_{\cG_{\phi}}(V) \cap \orbit^{+}_{\cG_{\phi}}(V)$ is open and dense in $N$, where 
$\cG_{\phi}= \cG(\phi_1, \phi_2)$. In other words, $\cG_{\phi}$ is transitive. This completes the proof of (1) in the case where $T={\rm id}$.

For an arbitrary recurrent $T$, let $\cR$ be the set of recurrent points of $T$, which is also invariant by $\phi_1$ and $\phi_2$. This set is dense. In fact, following an argument similar to the claim in the proof of Proposition  \ref{pro ifs 3}  this set is residual and of total Lebesgue measure.

Let $V$ be an open set in $N$, and $z\in \cR\cap \orbit^-_{\cG_{\phi}}(V)$. This intersection is obviously non-empty. Then, there are $d\in \NN$ and $\Sigma= (\sigma_1, \dots, \sigma_d), \sigma_i=1, 2$, such that 
$$z\in(\phi_{_{\Sigma}})^{-1}(V).$$
Moreover, for any $i=1, 2, \dots, d$, and any $l_j \in \ZZ$, $j=1, 2, \dots, i$,
$$\tilde{z}_i:= (T^{l_i}\circ\phi_{\sigma_i}) \circ \cdots  \circ  (T^{l_1}\circ\phi_{\sigma_1})(z) \in \cR.$$
So, using recurrence, for some (large) $l_j\in \NN$, the orbit $(\tilde{z}_i)_i$ shadows $(z_i)_i$, where $z_i=\phi_{\sigma_i} \circ \cdots  \circ \phi_{\sigma_1}(z)$.
This shows that for some  $l_j\in \NN$, the point $\tilde{z}_d$ belongs to $V$. But $(\tilde{z}_i)_i$ is an orbit of $z$ under the iterated function system of  
$$\cG_2= \cG(T, T\circ\phi_1, T\circ\phi_2).$$
In other words,  $\tilde{z}_d \in V \cap\orbit^{+}_{\cG_2}(z) $. Recall that $\cR\cap \orbit^-_{\cG_{\phi}}(V)$ is dense in $N$. So, the  $\cG_2$-orbit of any point in a dense set intersects $V$. The same is true for backward $\cG_2$-orbits. Thus, $\orbit^{\pm}_{\cG_2}(V)$ is (open and) dense in $N$, and $\cG_2$ is transitive. 
This completes the proof of (1).

Given $N_c \subset\subset N$ bounded, and $V \subset N$ open, 
we let $X= \overline{B_1(N_c)} \setminus \orbit^-_{\cG_2}(V)$. Observe that $X$ is a compact set with empty interior. So for any $x \in X$ there exists $h_x$ in $\Diff^r_\omega(N)$ and $\epsilon$-close to the identity in the $C^r$ topology, such that $h^{-1}_x(x) \in V^{-}:= \orbit^-_{\cG_2}(V)$. Since $V^{-}$ is an open set, there is a neighborhood $U_x$ of $x$ such that $h^{-1}_x(U_x) \subset V^{-}$. The family $\{U_x \}$ is an open cover of the compact set $X$. So there exist $k\in \NN, x_1, x_2, \dots, x_l \in X$ and $h_{x_1}, h_{x_2}, \dots, h_{x_k} \in B_{\epsilon}({\rm id}) \subset \Diff^r_\omega(N)$ such that
$$X \cap h^{-1}_{x_1}(X) \cap \cdots \cap h^{-1}_{x_k}(X)=\varnothing.$$

Thus  
$$T^{-1}(X) \cap (h_{x_1}\circ T)^{-1}(X) \cap \cdots \cap (h_{x_k}\circ T)^{-1}(X)=\varnothing.$$

Therefore, 
$$N_c \subset \subset T^{-1}(V^-) \cap (h_{x_1} \circ T)^{-1}(V^-) \cap \cdots \cap (h_{x_k}\circ T)^{-1}(V^-).$$

If we define $$\cG:=\cG(T, T\circ \phi_1, T\circ \phi_2, h_{x_1}\circ T, \dots, h_{x_k}\circ T),$$ 
then we have that
$N_c \subset \subset \orbit^-_{\cG}(V).$
\Endproof

\subsection{Skew products and IFS} \label{s skew & ifs}
In this section we explain the relation between iterated function systems
and skew products over shifts. This relation have been used extensively in random dynamical systems (cf. \cite{la}). Similar approach to the one we present here has been used by R. Moeckel in \cite{mo} where he studies the IFS of monotone  twist maps on annulus, relating it to the instability (drift) for skew product of such maps over shift. 

Here,  we state a series of propositions (see Propositions \ref{pro skew 3} and \ref{pro skew 4}) that describe  the trajectories of the action of a semigroup of diffeomorphisms as the dynamics of the unstable sets of an appropriate skew product over shifts.

Let $\tau$ be the full shift with $d$ symbols.
\begin{equation*}
\begin{aligned}
\tau:d^\ZZ & \to d^\ZZ    \\ 
x=(\dots,x_{-1},x_0;x_1,\dots) &\mapsto  (\dots,   x_0,x_1;x_2,\dots)
\end{aligned}
\end{equation*}
It is natural to define the local and global unstable manifolds of 
a point $x\in d^\ZZ$ for $\tau$ as the following
$$W^{u}_{loc}(x ;\tau)=\{(z_i)\mid \forall i\leq 0, z_i=x_i \}$$
$$W^{u}(x;\tau)
=\bigcup_{i\geq0} \tau^i(W^{u}_{loc}(\tau^{-i}(x) ;\tau))
=\{(z_i)\mid \exists i_0\in\mathbb{Z},\forall i\leq i_0, z_i=x_i\}.$$

Let $\Phi: d^\ZZ \x Y \to d^\ZZ \x Y$ be a skew product such that 
$$\Phi(x,y)= (\tau(x), \phi_{x}(y)),$$
where  $\phi_x$ is a homeomorphism on $Y$,
for any $x\in d^\ZZ$. 
Assume that the family of $\phi_x$'s are uniformly bi-Lipschitz, i.e., there exists 
$ L>1$ such that $ \forall x \in d^{\ZZ},~ \forall y, y' \in Y,$ 
$$  \frac{1}{L} {\rm dist}_Y(y, y') \leq {\rm dist}_Y( \phi_{x}(y), \phi_{x}(y')) \leq L  \cdot {\rm dist}_Y(y, y')$$ and let us assume that the Lipschitz constant varies continuously with respect to $x$.

We consider the following metric on $d^{\ZZ}$,
$${\rm dist}(x,z)=\sum_{i\in \ZZ} e^{-|i|L}|x_{i}-z_{i}| . $$
One may define the strong unstable manifold as follows:
$$W^{uu}(x,y;\Phi):=\{(a,b) \mid {\rm dist}(\Phi^{i}(x,y), \Phi^{i}(a,b) ) \sim \exp(i L) ~{\rm as}~ i\to-\infty\}.$$

Assume that $\phi_{x}$ depends only on $[x_{_{\leq i_{0}}}]:=(\dots,x_{i_{0}-1},x_{i_{0}})$, and denote it by  $\phi_{[x_{_{\leq i_{0}}}]}$. To avoid complications we also assume that $i_{0}=0$. 

Therefore, $\Phi$ is the product $\tau\x \phi_{x}$  on the set 
$\{z\in d^\ZZ \mid z_{i}=x_{i}, i\leq0 \} \x Y$. So, the local unstable manifold of
$(x,y)$ for $\Phi$ contains $W^{u}_{loc}(x ;\tau)\x\{y\}$. 
Then we have proved the following proposition.

\begin{proposition} \label{pro skew 1}
For any $(x,y)\in d^\ZZ \x Y$ and $n\in\NN$,
\begin{equation*}
\begin{aligned}
\Phi^n(x,y)= &(\tau^n(x),\phi_{[x_{_{\leq n-1}}]}\circ \cdots \circ \phi_{[x_{_{\leq 0}}]}(y)),\\
\Phi^{-n}(x,y) =& 
(\tau^{-n}(x),\phi^{-1}_{[x_{_{\leq-n}}]}\circ \cdots \circ \phi^{-1}_{[x_{_{\leq-1}}]}(y)).\\
W^{uu}_{loc}(x,y ;\Phi)= & W^{u}_{loc}(x ;\tau)\x\{y\}
= \{(z_i)\mid \forall i\leq 0, z_i=x_i\}\x\{y\},\\
W^{uu}(x,y;\Phi)= & \bigcup_{i\geq0}
\Phi^{i}(W^{uu}_{loc}(\Phi^{-i}(x,y);\Phi)).
\end{aligned}
\end{equation*}
\end{proposition}

Since $W^{uu}_{loc}(x,y;\Phi)$ is a product set and $\Phi$ is a product on it, so 
$\Phi^{i}(W^{uu}_{loc}(x,y;\Phi))$ is a finite union of some local strong unstable manifolds. Therefore, we have proved the following.

\begin{proposition} \label{pro skew 2}
For any $(x,y)\in d^\ZZ \x Y$, the global strong unstable manifold $W^{uu}_{loc}(x,y;\Phi)$ 
is a countable union of some local unstable manifolds $W^{uu}_{loc}((x^i,y^i);\Phi)$.
\end{proposition}

\subsubsection*{Locally constant skew products}
From now on we assume that
 $\phi_{x}$ depends only on $x_{0}$, where $x=(\dots,x_{-1},x_0;x_1,\dots)$.
Then, $\Phi=\tau\x\phi_{j}$ on the set
$\{z\in d^\ZZ \mid z_{0}=j \} \x Y$, for any $j\in\{1, 2, \dots, d\}$.
The next propositions shed some light on the relation between (locally constant) skew products over shifts and iterated function systems.

Let $\cG=\cG(\phi_1,\phi_2, \dots,\phi_d)$. 
Proposition \ref{pro skew 1} implies that for any $(x,y)\in d^\ZZ \x Y$
and $n\in\NN$,
$$\Phi^n(x,y)=(\tau^n(x),\phi_{x_{n-1}}\circ \cdots \circ \phi_{x_0}(y)).$$
This shows that  by taking different base points $x$, one can realize the orbit of $y$ under the IFS $\cG$. 
Since the skew product $\Phi$ does not depend on $x_{i}, ~ i>0$, so we get the entire positive $\cG$-orbit of $y$ by taking all points on $W^{u}_{loc}(x ;\tau)$. 
So, we have obtained the following proposition.

\begin{proposition} \label{pro skew 3}
For any $(x,y)\in d^\ZZ \x Y$, the projection of $\bigcup_{n>0}
\Phi^{n}(W^{uu}_{loc}(x,y;\Phi))$
on $Y$ is equal to 
$\orbit^+_\cG(\phi_{x_0}(y)).$
In particular, if $(x,y)$ is a fixed point of $\Phi$, then the projection of $W^{uu}(x,y;\Phi)$ on $Y$ is equal to $\orbit^+_\cG(y).$
\end{proposition}
As we said, this proposition turns out to be very useful in the study of dynamical
properties of strong stable/unstable manifolds of certain partially
hyperbolic skew product systems. For instance, to obtain the density of the strong
unstable manifolds (see Section \ref{s min thm A}).  More precisely
let us mention the geometric meaning of this fact: 
by each iteration the length of $W^{uu}_{loc}(x,y;\Phi)$ grows exponentially and it intersects the domain of all $\phi_i$'s. Therefore, all possible compositions of $\phi_i$'s do appear in the positive orbit of $W^{uu}_{loc}(x,y;\Phi)$.  Indeed, we have the following proposition which gives a precise description of the global strong unstable manifolds for $\Phi$. 

\begin{proposition} \label{pro skew 4}
For any $(x,y)\in d^\ZZ \x Y$,
$$W^{uu}(x,y;\Phi)= \bigcup_{\sigma\in\Sigma}
W^{uu}_{loc}(x^\sigma,\phi^{x,\sigma}(y);\Phi),$$
where  \vspace{-.2cm} 
$$\Sigma=\{\sigma=(\sigma_1, \dots, \sigma_n)\mid  n\in\NN, 1\leq\sigma_i\leq d\},$$
$$\phi^{x,\sigma}=\phi_{\sigma_{n-1}}\circ \cdots \circ \phi_{\sigma_{1}}\circ 
\phi^{-1}_{x_{-n+1}}\circ  \cdots \circ \phi^{-1}_{x_{-1}}  \andrm  \phi^{x,(\sigma_1)}={\rm id},$$
$$x^\sigma=(\dots, x_{-n},\sigma_1, \dots, \sigma_n; x_1,\dots).$$
\end{proposition}

\Proof
It is easy to see that for  any $\sigma, \sigma' \in\Sigma$, 
$W^{u}_{loc}(x^{\sigma};\tau)=W^{u}_{loc}(x^{\sigma'} ;\tau)$
if and only if $\sigma=\sigma'$. Moreover, $\tau^n(x^{\sigma})\in W^{u}_{loc}(x ;\tau)$ if $\sigma=(\sigma_1, \dots, \sigma_n)$.
Therefore,
$$W^{u}(x;\tau)=\bigcup_{\sigma\in\Sigma} W^{u}_{loc}(x^{\sigma};\tau).$$

On the other hand, the projection of $W^{uu}(x,y;\Phi)$
on $d^\ZZ$ is equal to  $W^{u}(x ;\tau)$, since $\Phi$ is skew
product.

$$W^{uu}(x,y;\Phi)=\bigcup_{n\geq0}
\Phi^{n}(W^{uu}_{loc}(\Phi^{-n}(x,y);\Phi)).$$

\begin{equation*}
\begin{aligned}
\Phi(W^{uu}_{loc}(\Phi^{-1}(x,y);\Phi)) &
=\Phi(W^{u}_{loc}(\tau^{-1}(x);\tau)\x\{\phi^{-1}_{x_{-1}}(y)\}) \\ &
=(\tau\x \phi_{x_{-1}})(W^{u}_{loc}(\tau^{-1}(x);\tau)\x\{\phi^{-1}_{x_{-1}}(y)\}) \\ &
= \tau(W^{u}_{loc}(\tau^{-1}(x);\tau))\x\{y\}  \\ &
=\bigcup_{|\sigma|=1}W^{uu}_{loc}(x^\sigma,y;\Phi).
\end{aligned}
\end{equation*}

From the definition of global unstable manifolds and Proposition \ref{pro skew 1} it follows that for any $n\in\NN$,
\begin{equation*}
\begin{aligned}
\Phi^{n}(W^{uu}_{loc}(p,q;\Phi)) &
= \bigcup_{a\in W^{u}_{loc}(p;\tau)}\{(\tau^{n}(a), \phi_{a_{n-1}}\circ \cdots \circ \phi_{a_0}(q))\} \\ &
= \bigcup_{\substack{\eta=(p_{0},a_{1}, \dots, a_{n-1}) \\ 
a= (...,p_{0};a_{1},a_{2},...)} }
\{(\tau^{n}(a),  \phi_\eta(q))\}
\end{aligned}
\end{equation*}
Now let $(p,q)=\Phi^{-n}(x,y)$, then $p_{i}=x_{i-n}, ~ (\forall i\in\ZZ)$, and 
$q=\phi^{-1}_{x_{-n}}\circ \cdots \circ \phi^{-1}_{x_{-1}}(y) $.
Thus, 
$\tau^{n}(a)=(...,p_{0},a_{1},\dots,a_{n};a_{>n})=(...,x_{-n},a_{1}, \dots ,a_{n};a_{>n})$ and  $ \eta=(x_{-n},a_{1}, \dots, a_{n-1})$.
Therefore,
\begin{equation*}
\begin{aligned}
\phi_\eta(q) &
=\phi_{\eta_{n}}\circ \cdots \circ \phi_{\eta_{1}}\circ 
\phi^{-1}_{x_{-n}}\circ  \cdots \circ \phi^{-1}_{x_{-1}}(y)  \\ &
=\phi_{a_{n-1}}\circ \cdots \circ \phi_{a_{1}}\circ \phi_{x_{-n}}\circ \phi^{-1}_{x_{-n}} \circ
\phi^{-1}_{x_{-n+1}}\circ  \cdots \circ \phi^{-1}_{x_{-1}}(y) \\&
=\phi_{a_{n-1}}\circ \cdots \circ \phi_{a_{1}}\circ 
\phi^{-1}_{x_{-n+1}}\circ  \cdots \circ \phi^{-1}_{x_{-1}}(y).
\end{aligned}
\end{equation*}

It yields that,
\begin{equation*}
\begin{aligned}
\Phi^n(W^{uu}_{loc}(\Phi^{-n}(x,y);\Phi)) &
= \bigcup_{\substack{\eta=(p_{0},a_{1}, \dots, a_{n-1}) \\
a= (...,p_{0};a_{1},a_{2},...)} }
\{(\tau^{n}(a),  \phi_\eta(q))\} \\ &
= \bigcup_{|\sigma|=n}W^{uu}_{loc}(x^\sigma,\phi^{x,\sigma}(y);\Phi).
\end{aligned}
\end{equation*}
This completes the proof.
\Endproof

The stable manifolds of these maps are defined as the unstable manifolds
of corresponding inverse maps. Similar results hold for stable manifolds.

\section{Blenders, double blenders and symplectic blenders} \label{sec blender}

The definition, existence and properties of blenders,  double blenders and symplectic blenders are discussed in this section.

Bonatti and D{\'\i}az in \cite{bd} introduced the notion of blender, which is a geometric model for certain hyperbolic sets that play an important role as a mechanism for creation of cycles and a semi-local source of transitivity. Although their methods may be easily modified for the conservative case, the symplectic one is more delicate. 

In \cite{bd} a $cu$-blender, roughly speaking, is a hyperbolic (locally maximal) invariant set with a splitting of the form $ E^{ss} \oplus E^{u} \oplus E^{uu} $, dim$E^{u}=1 $, such that a convenient projection of its stable set has larger topological dimension than the dimension of the subbundle (cf. Lemma \ref{lem magic}). This phenomenon is robust in the $C^{1}$ topology. Similarly, one may define a $cs$-blender.

Their constructions  essentially use a hyperbolic set with a
one-dimensional weakly hyperbolic subbundle.
On the other hand, to apply this local tool for systems with higher dimensional central bundles they use a chain of blenders with one-dimensional central bundles with different indices (i.e. dimension of the stable bundle) connected to each other. This allows them to use such blenders in more general situations. This approach,  of course, is impossible in the symplectic case, since all eigenvalues are pairwise conjugate and so all hyperbolic periodic points have the same index. \textit{So, to build blenders in the symplectic case, we need to involve higher central dimensions.} We construct a new class of such blenders in the symplectic (or Hamiltonian) systems that works like a chain of $cs$-blenders and a chain of $cu$-blenders simultaneously.

In Section \ref{s symbo blender}, regardless of the symplectic case, we give the definition of $cs$-blenders and we study an abstract model for blenders with arbitrary central dimensions. Using the inverse map, we can  define the $cu$-blenders (see Section \ref{s cu-blender}). For their construction we use the results in Section \ref{s skew & ifs} and the ones about contracting (expanding) IFS stated in Section \ref{s ifs contract}. Then, in Section \ref{s geometric cs-blender}  we consider a geometric model of $cs$-blender and prove in Section \ref{s geometric cs blender are cs blender} that their main feature of being a blender is a robust property. To do that, first we recast the hypothesis of Proposition \ref{pro ifs 2} for IFS  in terms of properties of the geometric model of $cs$-blender (see Section \ref{s properties cs-blender}).

In Section \ref{s double blender}, we consider the case where the central bundle 
splits into two stable and unstable subbundles, that is, the maximal invariant
set is hyperbolic of the form 
$ E^{ss} \oplus E^{s} \oplus E^{u} \oplus E^{uu}$. We call it \textit{double-blender}.
Then we introduce an abstract model which exhibits the feature of 
both $cu$- and $cs$-blenders (see Section \ref{s symbo double blender}). Note that this case is very compatible with the symplectic case where the eigenvalues of periodic points are pairwise conjugate. 

In Section \ref{s blender}, we introduce the above phenomenon in the context of symplectic diffeomorphisms. That is what we call \textit{symplectic blender}. Moreover, in Theorem \ref{thm blender} we show how the symplectic blenders appear in the context of Theorem \ref{thm A} .  In Section \ref{blender in ph} we give a series of results about the role of blenders inside partially hyperbolic sets. These results are essential in the proof of our mains theorems.


\subsection{The $cs$-blenders}\label{s cs-blender}
In this section we introduce the definition of $cs$-blender.

\begin{definition} ($s$-strip)
Let $F$ be a diffeomorphism on the manifold $M$.
Let $\cB$ be an open embedded ball  with three cone-fields $\mathcal{C}^{ss}$, 
$\mathcal{C}^{s}$, $\mathcal{C}^{u}$, invariant under the derivative $DF$ defined in a compact neighborhood of $\cB.$ Observe that  the cone fields induce three invariant subbundles $E^{ss}, E^s, E^u$.
A horizontal strip (or $s$-strip) is an embedded ($s+ss$)-dimensional disk in $\cB$ tangent to $E^{ss}\oplus E^s$ and  
which contains the $ss$-leaves of each its points. 
\end{definition}

\begin{definition}[ $cs$-blender] \label{def cs-blender}
The pair  $(P,\cB)$ is a $cs$-blender for the diffeomorphism $F$ if it satisfies
the following features:
\begin{list}{\textbf{B}-\arabic{Lcount} $\;$}{\usecounter{Lcount}}
\item $P$ is a hyperbolic saddle periodic point of $F$ contained in $\cB$;
\item $\cB$ is an open embedded ball on which there are three hyperbolic cone fields $\mathcal{C}^{ss}$, 
$\mathcal{C}^{s}$ and $\mathcal{C}^{u}$ invariant under the derivative $DF$.
\item Any $G$ sufficiently close to $F$ in the $C^1$ topology verifies that any $s$-strip in $\cB$ intersects  the unstable manifold of $P_G$ whose backward orbit is in $\cB$. Here $P_G$ is the continuation of $P$. 
\end{list}
\end{definition}

From the definition of $cs$-blender  follows immediately the next lemma (the proof is left to the reader).
\begin{lemma}\label{lem cs-blender def} Let  $(P,\cB)$ be a $cs$-blender. Let $B_G=\bigcap_{n\in \ZZ}G^n(closure(\cB))$ for $G$ close to $F.$ Then it follows that $B_G$ is a hyperbolic set such that for any 
$s$-strip $S$ in $\cB$ there is $x\in B_G$ such that $W^u_{loc}(x)\cap \cB\neq \vazio$, where $ W^u_{loc}(x)$ denote the local unstable manifold of $x.$
 
\end{lemma}

\begin{remark}\label{rmk cu blender 1}
 Using the inverse map, and defining properly  a $u$-strip (see Definition \ref{defi u strip}), we can define the notion of $cu$-blender (see details in Section \ref{s cu-blender}). In this case, the three hyperbolic cone fields are $\mathcal{C}^{s}$, 
$\mathcal{C}^{u}$, $\mathcal{C}^{uu}$  and it is required that the stable manifold of the saddle periodic point intersects any $u$-strip.
 
\end{remark}

\subsubsection{Symbolic $cs$-blender}\label{s symbo blender}
In this section we use the language of Section \ref{s skew & ifs} to introduce a symbolic interpretation and construction of $cs$-blenders.

Let  $\mathfrak{S}$ be the space of all skew products  $\Phi: k^\ZZ \x \RR^n \to k^\ZZ \x \RR^n$ such that 
$$\Phi(x,y)= (\tau(x), \phi_x(y)),$$
where $\tau:k^\ZZ\to k^\ZZ$ is the full shift with $k$ symbols and for any $x\in k^\ZZ$, $\phi_x=\phi_{W^{u}_{loc}(x)}$ is a contracting map on $\RR^{n}$ with positive contraction bound. 

Let $\Phi=(\tau,\phi_{x}), \Psi=(\tau, \psi_{x}) $ in $\mathfrak{S}$. 
We say that $\Phi$ is close to $\Psi$ if for any $x\in k^{\ZZ}$,  $\phi_{x}$ and its contraction bound  are  close to $\psi_{x}$  and its contraction bound, respectively.

By a  $s$-strip we mean the set $W^{s}_{loc}(x;\tau)\x U$ for some $x\in k^\ZZ$ and some open set $U\subset\RR^{n}$.

\begin{definition}\label{def symb blender}
Let $D$ be an open set in $\RR^n$.  The set $\cB=k^\ZZ\x D$ is a symbolic $cs$-blender of $\Phi\in \mathfrak{S}$ if 
 there exists a fixed point $(p,q)$ of $\Phi$ such that  for any $\tilde\Phi \in \mathfrak{S}$ close to $\Phi$ the following hold:

\begin{enumerate}
\item $(p,q)$ has a unique continuation $(p,\tilde q)$ for $\tilde\Phi$,
\item the unstable manifold of $(p,\tilde q)$ for $\tilde\Phi$ 
intersects any $s$-strip in $\cB$.
\end{enumerate}
A symbolic $cu$-blender is defined as  a symbolic $cs$-blender for $\Phi^{-1}$.
\end{definition}

Observe that the existence of a symbolic $cs$-blender is not trivial, specially its robustness property.
The next proposition guarantees the existence of a symbolic $cs$-blender.
 
\begin{proposition} \label{pro skew model}
Let $\Phi(x,y)= (\tau(x), \phi_{x_0}(y))$ be a locally constant  skew product, where $x_0\in \{1,\dots, k\}$ and $x=(\dots,x_{-1},x_0;x_1,\dots)$. Assume that $\phi_i: \RR^n \longrightarrow \RR^n$, $i=1, 2, \dots , k$, are contracting maps, and $\phi_i(z_i)=z_i$ are their unique fixed points. Suppose that the iterated function system $\cG=\cG(\phi_1, \dots , \phi_k )$ has  the covering property on $D \subset \RR^n$ and the set $\{z_i\}_{i=1}^k$ is well-distributed in $D$.
Then the set $\cB=k^\ZZ \x D$ is a symbolic $cs$-blender of $\Phi$.
\end{proposition}
\Proof
First we are going to prove the properties of the $cs$-blender for the initial system and latter we deal with the perturbations.

Let $(p,q)$ be a fixed point of $\Phi$, for instance,  $p=(\dots, 1, 1; 1,1, \dots)$ and $q=z_{1}$ is the unique fixed point of $\phi_{p}=\phi_{1}$.

Recall that the unstable manifold of $(p,q)$ is defined as the set of all points $(x,y)$ such that ${\rm dist}(\Phi^{i}(x,y), \Phi^{i}(p,q) ) \to 0 ~{\rm as}~ i\to-\infty $. Since the dynamic on $\RR^{n}$ is contracting, then the unstable manifold of $(p,q)$ coincides with $W^{uu}(p,q; \Phi)$. 
By Proposition \ref{pro skew 3}, the projection of $W^{uu}(x,y;\Phi)$ on $\RR^{n}$ 
is equal to $\orbit^+_\cG(y)$. 
Proposition \ref{pro ifs 2} yields that $D \subset \overline{\orbit_{\cG}^+(q)}$.
In other word, the unstable manifold of $(p,q)$ for $\Phi$ intersects any $s$-strip in $\cB$.

Now we show the robustness of this property. Let $\tilde\Phi \in \mathfrak{S}$ be close to $\Phi$. Since   $\tilde\phi_{p}$ is a contracting map close to $\phi_{p}=\phi_{1}$, it follows that  it has a unique fixed point $\tilde q$  close to $q$. We want to show that $W^{uu}(p,\tilde q; \tilde\Phi)$ is dense.
 Observe that each $\tilde\phi_x$ is close to $\phi_{x_0}.$ Now, given $x_0\in \{1,\dots, k\}$ we take  $\bar x_0=(\dots,x_0,x_0;x_0\dots)$ and we consider the IFS $\tilde\cG=\cG(\{\tilde\phi_{\bar x_0}\}_{x_0\in\{1,\dots, k\}})$.
The proof of  the density of the unstable manifold of $(p, \tilde q)$ follows from the density of the $\tilde\cG$-orbit of $\tilde q$ which holds  from the robustness of $\cG$ (see Propositions \ref{pro ifs 1} and  \ref{pro skew 3}). This yields that the unstable manifold of $(p,\tilde q)$ for $\tilde\Phi$ 
intersects any $s$-strip in $\cB$.
\Endproof


\subsubsection{Geometric model of a symbolic $cs$-blender}\label{s geometric cs-blender}
Using the skew product construction it is easy to build a geometric model of the symbolic blender. 

Let  $\phi_i: \RR^n \longrightarrow \RR^n$, $i=0, 1, \dots , k$, be  an affine contracting map as in  Proposition \ref{pro ifs 1}. Let $D$ be an open disk in $ \RR^n $ such that $\phi_i(D)\subset D$.

Let $f:\RR^{m} \to \RR^{m}$ be a diffeomorphism with a horseshoe type hyperbolic set $\Lambda=\bigcap_{n\in\ZZ} f^{n}(U)$, and with the splitting of the form $E^{ss}\oplus E^{u}$.  
 Then it has a Markov partition with $k+1$ symbols by rectangles $R_{0}, \dots, R_{k}$ such that for any $i\neq j$,  $~ \overline{R_{i}}\cap \overline{R_{j}}=\vazio$.  
 We also assume that the contraction rate on $E^{ss}$ is stronger than the contraction of $\phi_{i}$'s. 

We take a diffeomorphism  $F:\RR^{m}\x\RR^{n} \to \RR^{m}\x \RR^{n}$ such that $$F|_{R_{i}\x\RR^{n} }= f\x\phi_{i}.$$
Observe that $H:=\bigcap_{n\in \ZZ} F^n(U\x D)$ is a hyperbolic basic set with splitting $E^{ss}\oplus E^{s} \oplus E^u.$

Now we apply Proposition \ref{pro skew model} for $F$.  It yields that the set  $\Lambda\x D$ is a symbolic blender with a fixed point $(p,q)$ of $F$. 

The following lemma is an immediate consequence of the definition of blender. It highlights the distinguished property of a blender. 

\begin{lemma} \label{lem magic}
Let $\Lambda\x D$ be the geometric model of the  symbolic $cs$-blender as above. 
Then, for any $(x_1,x_2)\in U\x D$, there exists $(y_1,y_2)\in H$ such that 
$$W^{ss}_{loc}(x_1,x_2; F)\cap W^{u}_{loc}(y_1,y_2; F)\neq \vazio.$$
\end{lemma}

This lemma says that the topological dimension of $W^{u}_{loc}(H)$ is  larger than the topological dimension of the local unstable manifold of each point in $H$.

\Proof
Let $\{U_n\}_{n\in\NN}$  be a sequence of nested open neighborhood of $x_2$, such that $\bigcap_{n\in\NN} U_n = \{x_2\}$. 

By definition, for any $n$, $W^{s}_{loc}(x_1;f)\x U_n$ is a $s$-strip containing $W^{ss}_{loc}(x_1,x_2;F)$ and  intersecting $W^{u}(p,q; F)$.  Now, let ${\bf z}_{n}\in H\cap W^{u}(p,q; F)$ such that 
$$(W^{s}_{loc}(x_1;f)\x U_n) \cap   W^{u}_{loc}({\bf z}_{n}; F)\neq \vazio.$$
Let ${\bf z}$ be an accumulation point of the sequence $\{ {\bf z}_n\}$. Then ${\bf z} \in H $. Now, let  $n\to \infty $, then from the election of the $s$-strips (i.e. $\bigcap_{n\in\NN} U_n = \{x_2\}$) it follows  that 
$W^{ss}_{loc}(x_1,x_2,F) \cap W^u_{loc}({\bf z}; F)\neq \vazio.$
\Endproof

The following remark allows to recast the essential property of a $cs$-blender.

\begin{remark}\label{rmk cs-blender ph 1} Given a blender $(p, \cB)$ and its maximal invariant set $B$, 
from Lemma \ref{lem blender def} it follows that for any $x\in \cB$ holds that $W^{ss}_{loc}(x)\cap W^u_{loc}(B)\neq \vazio$.
Moreover, the same holds for any $G$ close to $F.$
\end{remark}

\subsubsection{Covering and well distributed properties of the geometric model of a $cs$-blender}\label{s properties cs-blender}

Now we would like to recast the hypothesis of Proposition \ref{pro ifs 2}, i.e., the {\it covering and well distributed properties} provided in Definition \ref{prop cont IFS},  in terms of properties of the geometric model of a $cs$-blender. Observe that the Markov partition $R_{0}, \dots, R_{k}$ for $\Lambda$ and given  by rectangles in $\RR^m$, provides a Markov partition $\hat R_{0}, \dots, \hat R_{k}$ for $\Lambda\x D$ in $\RR^m\x \RR^n$ where $\hat R_i=R_i\x D.$ 

Let $\cD$ be the disk  that  is covered by the union of $\phi_i(\cD)$. Let us also define $R_{\cD}:=\cup_i R_i \x \cD$. Now, given a point $x$ in $\hat R_i$ we define $W^{ss}_{\hat R_i}(x)$ as the connected component of $W^{ss}(x)\cap \hat R_i$ that contains $x.$ Then, it holds that for any $x\in R_{\cD}$ and taking $j $ such that $x\in \hat R_j$ it follows that 
\begin{eqnarray}\label{covering 1} W^{ss}_{\hat R_j}(x)\cap \cup_i F(R_{\cD}\cap \hat R_i)\neq \emptyset, \end{eqnarray}
 Moreover, observe that from the Markov property holds that $\cD\subset  \cup_i \phi_i(\cD)$, it follows that if $W^{ss}_{\hat R_j}(x)$ intersect  $F(R_{\cD}\cap \hat R_i)$ then 
\begin{eqnarray}\label{covering 2}F^{-1}(W^{ss}_{\hat R_j}(x)\cap F(R_{\cD}\cap \hat R_i))\cap \cup_i F(R_{\cD}\cap \hat R_i)\neq \emptyset.
 \end{eqnarray}
Conditions (\ref{covering 1}) and (\ref{covering 2}) are the {\it  covering property} for the geometric model.

The {\it well-distributed} property is recasted in the following way: given  $z_0,\cdots, z_k$  the set of  unique fixed point of $F$ in each $\hat R_i,$  if any open ball of diameter $d$ and centered in $R_{\cD}$ contains some $z_i$, where 
$$d\geq\max\{r\mid \forall x\in R_{\cD}, \exists i, B_r(x) \subset F(R_i\x \cD) \}.$$

\subsubsection{Geometric model of a symbolic $cs$-blender is a $cs$-blender}\label{s geometric cs blender are cs blender}
In this section we show that the property that the unstable manifold of any point intersects any $s$-strip also holds for any perturbation.
It is natural to think that any small perturbation of a geometric model of a symbolic $cs$-blender is also a geometric model of a symbolic $cs$-blender, in particular, that it is a skew-product locally constant over contracting IFS. But this is not generically true. In fact, if (under a change of coordinates if necessary) we get a skew-product locally constant  over contracting IFS, it would follow that the laminations associated to the strong stable and unstable subbundles are Lipschitz, which is generically false. So, to prove that  perturbations of a geometric model verify condition {\bf{B}}-3 in Definition \ref{def cs-blender} we can not reduce to skew-product locally constant. However, the proof of Proposition \ref{pro ifs 2} is easily adapted to the context of perturbations of the geometric model of a symbolic $cs$-blender, proving in this way that they are $cs$-blenders. 

The proof runs as follows: First, given $G$ $C^1$-close to $F$ in a neighborhood of $\Lambda\x D$,  we consider the set $B_G$ which is the maximal invariant set of $G$ in that neighborhood. 
Observe that the covering  and well-distributed property also holds for $B_G$. This is immediate from the fact that $G( R_{\cD}\cap \hat R_i)$ is $C^0$-close to $F(R_{\cD}\cap \hat R_i)$ and $W^{ss}_{\hat R_j, G}(x)$ is $C^1$-close to  $W^{ss}_{\hat R_j, F}(x)$. So, given $x$ in $R_{\cD}$ we consider a $s$-strip $S$ around $W^{ss}_{\hat R_i}(x, G)$ (where $i$ is such that $x\in \hat R_i$)  and we take $S \cap G(\hat R_j \cap R_{\cD})$, where $j$ verifies that $S \cap G(\hat R_j \cap R_{\cD})\neq \emptyset.$ Observe that $G^{-1}(S\cap G(\hat R_j \cap R_{\cD}))$ is again an $s+ss$-embedded  disk tangent to $E^{ss}\oplus E^s$  that intersects $R_{\cD}$ and therefore we can choose new indexes $i, j$ as before; since the $DG^{-1}_{|E^{ss}\oplus E^s}$ is expanding, in the same way that in Proposition \ref{pro ifs 2} it follows that a backward orbit of a subdisk inside $S$ eventually intersects the local unstable manifold of some fixed point ${z_i}_G$ (where  ${z_i}_G$ is the analytic continuation of some $z_i$). Since all the fixed points are homoclinic connected for $G$, we  prove in that way that the unstable manifold of any fixed point intersects any $s$-strip. This finishes the proof.

\subsection{The $cu$-blenders}\label{s cu-blender}

Using the inverse map, we can introduce the notions of  $cu$-blender and symbolic $cu$-blender, its geometric model and also the properties of covering and well distribution.
In the same way we see that the geometric models of symbolic $cu$-blenders are in fact $cu$-blenders. The details are left for the reader.


\subsection{Double-blender}\label{s double blender}
In this section we introduce the definition of double blender.

\begin{definition}\label{defi u strip} ($u$-strip)
Let $F$ be a diffeomorphism on the manifold $M$.
Let $\cB$ be an open embedded ball  with four cone-fields $\mathcal{C}^{ss}$, 
$\mathcal{C}^{s}$, $\mathcal{C}^{u}$, $\mathcal{C}^{uu}$, invariant under the derivative $DF$ defined in a compact neighborhood of $\cB.$
A vertical strip (or $u$-strip) is an embedded ($u+uu$)-dimensional disk in $\cB$,
which contains the $uu$-leaves of each its points. 
\end{definition}

\begin{definition}[double blender] \label{def blender}
The pair  $(P,\cB)$ is a double blender for the diffeomorphism $F$ if it satisfies
the following features:
\begin{list}{\textbf{B}-\arabic{Lcount} $\;$}{\usecounter{Lcount}}
\item $P$ is a hyperbolic saddle periodic point of $F$ contained in $\cB$;
\item $\cB$ is an open embedded ball on which there are four hyperbolic cone fields $\mathcal{C}^{ss}$, 
$\mathcal{C}^{s}$, $\mathcal{C}^{u}$ and $\mathcal{C}^{uu}$ invariant under the derivative $DF$  defined in a compact neighborhood of $\cB.$
\item Any $G$ sufficiently close to $F$ in the $C^1$ topology verifies the following:

\begin{enumerate}
 \item any $u$-strip in $\cB$ intersects
some  $s$-strip  contained in the  stable manifold of $P_G$ whose  forward orbit is in $\cB$;

\item any $s$-strip in $\cB$ intersects some $u$-strip contained in the unstable manifold of $P_G$ whose backward orbit is in $\cB$.
\end{enumerate}

Here $P_G$ is the continuation of $P$. 
\end{list}
\end{definition}

The next lemma follows immediately from Definition \ref{def blender} and the proof of Lemma \ref{lem magic}.
\begin{lemma}\label{lem blender def} Let  $(P,\cB)$ be a double blender. Let $B_G=\bigcap_{n\in \ZZ}G^n({\rm closure}(\cB))$ for $G$ close to $F.$ Then it follows that $B_G$ is a hyperbolic set such that for any 
$s$-strip $S$ in $\cB$ there is $x\in B_G$ such that $W^u_{loc}(x)\cap \cB\neq \vazio$, where $ W^u_{loc}(x)$ denotes the local stable manifold of $x.$
A similar statement holds for any $u$-strip.
 
\end{lemma}

The following remark allows to recast the essential property of a blender.

\begin{remark}\label{rmk blender ph 1} Given a blender $(p, \cB)$ and its maximal invariant set $B$, 
from Lemma \ref{lem blender def} it follows that for any $x\in \cB$ holds:
 \begin{itemize}
\item [-] $W^{ss}_{loc}(x)\cap W^u_{loc}(B)\neq \vazio$,

\item [-] $W^{uu}_{loc}(x)\cap W^s_{loc}(B)\neq \vazio$.
\end{itemize}
Moreover, the same holds for any $G$ close to $F.$
\end{remark}

\subsubsection{Symbolic double-blender}\label{s symbo double blender}
Here we introduce an abstract model with the features of symbolic $cu$- and $cs$-blenders, simultaneously. We call it  a symbolic double-blender.

Let  $\mathfrak{A}$ be the space of  $\Phi: k^\ZZ \x \RR^n\x \RR^m \to k^\ZZ \x \RR^n\x \RR^m$ be a skew product such that 
$$\Phi(x,y,z)= (\tau(x), \phi_x(y), \psi_x(z)),$$
where $\tau:k^\ZZ\to k^\ZZ$ is the full shift with $k$ symbols, for any $x\in k^\ZZ$, $\phi_x=\phi_{_{W^{u}_{loc}(x)}}$ is  contracting map on $\RR^{n}$ with contraction bound larger than a fixed positive number, and $\psi_x=\psi_{_{W^{s}_{loc}(x)}}$ is  an expanding map on $\RR^{m}$ with expansion bound smaller than a fixed number.

Let  us fix $\Phi=(\tau,\phi_{x}, \psi_{x}), \tilde\Phi=(\tau,\tilde\phi_{x},\tilde\psi_{x} ) $ in $\mathfrak{A}$. 
We say that $\Phi$ is close to $\tilde\Phi$ if for any $x\in k^{\ZZ}$,  $\phi_{x}$ and its contraction bound  are  close to $\tilde\phi_{x}$  and its contraction bound, respectively, and  $\psi_{x}$ and its expansion bound  are  close to $\tilde\psi_{x}$  and its expansion bound, respectively.

For $\Phi=(\tau,\phi_{x}, \psi_{x})$ in $\mathfrak{A}$, we denote $\Phi_{cs}=(\tau,\phi_{x})$, and $\Phi_{cu}=(\tau, \psi_{x})$. 

\begin{definition}
Let $U$ and $V$ be open sets in $\RR^n$ and $\RR^m$, respectively. The set $\cB=k^\ZZ\x U\x V$ is a symbolic double-blender of $\Phi\in \mathfrak{A}$ if 
$k^\ZZ\x U$ is a $cs$-blender of $\Phi_{cs}$ and $k^\ZZ\x V$ is a $cu$-blender of $\Phi_{cu}$. 
\end{definition}
 
A similar proposition to the one formulated for a symbolic cs-blender (Proposition \ref{pro skew model}) can be formulated in the context of a symbolic double blender.
 
\begin{proposition} \label{pro double model} 
Let $\Phi\in\mathfrak{A}$ be a locally constant  skew product such that $\Phi(x,y,z)= (\tau(x), \phi_{x_{0}}(y), \psi_{x_{1}}(z))$, where $x=(\dots,x_{-1},x_0;x_1,\dots)\in k^{\ZZ}$. Assume that $\phi_i$ and $\psi_{i}^{-1}$, $i=1, 2, \dots , k$, are contracting maps.
Suppose that the iterated function systems $\cG=\cG(\phi_1, \dots , \phi_k )$ and $\cG=\cG(\psi_1^{-1}, \dots , \psi_k^{-1} )$ have covering and well-distributed properties on $D_{1} \subset \RR^n$ and $D_{2} \subset \RR^m$, respectively.
Then the set $\cB=k^\ZZ \x D_1\x D_{2}$ is a symbolic double-blender of $\Phi$.
\end{proposition}
 
The proof is similar to the proof of Proposition \ref{pro skew model} and it is left to the reader.


\subsubsection{Geometric model of symbolic double-blender}
Using the skew product construction it is easy to build a geometric model of a symbolic double-blender. The idea is the same exposed in Section \ref{s symbo blender} where the geometric model of a symbolic $cs$-blender was considered. But instead of considering a set of affine contracting maps we use a set of pairs of affine maps, one contracting and one expanding. 
In few words,  we consider the skew product $F$ such that $F|_{R_{i}\x\RR^{n}\x\RR^{n} }= f\x\phi_{i}^s\x\phi_i^u$, where for any $i$, $\phi^s_i:\RR^n\to \RR^n$ is contracting and $\phi^u_i:\RR^n\to \RR^n$ is expanding.


 The geometric idea behind this definition is the following. 
In the 3-dimensional $cs$-blenders  of \cite{bd}, if one projects the cube and its pre-image along of stable direction, a figure like a Smale horseshoe appears but the right and left rectangles overlap. 
With this in mind, consider a 4-dimensional horseshoe with the splitting of the form $ E^{ss} \oplus E^{s} \oplus E^{u} \oplus E^{uu} $  where the projection along $E^{ss}$ gives a figure like a 3-dimensional horseshoe such that its two wings overlaps and the same feature holds for the inverse map. 
Let $F$ be such a diffeomorphism on the open set $\cB$.
Then, the maximal invariant set in $\cB$, i.e., $\Lambda=\bigcap_n F^{n}(\cB)$ is a $cs$-blender, if we consider $E^{ss} \oplus E^{s}$ as the stable direction, $E^{u}$ as the central unstable direction and $E^{uu}$ as the strong unstable direction. Similarly $\Lambda$ is a $cu$-blender if we consider $E^{ss}$ as the strong stable direction, $E^{s}$ as the central stable direction and $E^{u} \oplus E^{uu}$ as the unstable direction. Therefore, $\Lambda$ is a double-blender. Note that using the results in Section \ref{s symbo blender}, we may consider multi-dimensional central bundles, i.e., both of center stable and center unstable bundles of arbitrary dimensions.

In a similar way as it was done in Section \ref{s geometric cs-blender} the notion of {covering and well distributed properties}, can be formulated for double-blenders.


\subsubsection{Geometric model of a symbolic double-blender is a double-blender}\label{s geometric double blenders are double blender}
The proof is essentially the same as the one provided for $cs$-blender in Section \ref{s geometric cs blender are cs blender} but working simultaneously along the center stable and center unstable directions.

\subsection{Symplectic blender} \label{s blender}

\begin{definition}
A \textit{symplectic blender} is a double-blender for a symplectic (or Hamiltonian) diffeomorphism. 
\end{definition}
Observe that if we want to consider geometric models of symplectic blenders as it was done for the double-blenders, it is necessary to chose the pairs of contracting and expanding affine maps, in such a way that a symplectic relation is satisfied. For instance, $\phi_i^u={(\phi^s_i)}^{-1}$.

The following theorem introduces a construction which yields the existence of a symplectic blender in the context of Theorem \ref{thm A}. 
 
\begin{theorem}\label{thm blender}
Let $M$ and $N$ be two symplectic manifolds (not necessarily compact). Let $r=1, 2, \dots, \infty$. Suppose that $f_{1}\in \mathrm{Diff}^{r}_{\omega}(M)$ has a hyperbolic periodic point $\hat{p}$ with transversal homoclinic intersections and $f_{2}\in \mathrm{Diff}^{r}_{\omega}(N)$  has a $\delta$-weak hyperbolic periodic point $\hat{q}$ with $\delta=\delta(f_1,\dim(N))>0$ small enough.

Then, there is a $C^{r}$ arc $\{F_{\mu}\}_{\mu\in [0,1]}$ of $C^{r}$ symplectic diffeomorphisms on $M\times N$ such that,
\begin{enumerate}
\item $F_{0}=f_{1}\times f_{2}$.
\item There is a neighborhood $\cV$ of $\{F_{\mu}\}_{\mu\in (0,1]}$ in $\Diff^1(M\x N)$ such for any $G\in \cV$, the pair $(P_G, \cB)$ is a double blender, where $P_G$ is the continuation of hyperbolic $P_0=(\hat{p}, \hat{q})$ and $\cB$ is an embedded open disk in $M\x N$.
\end{enumerate}
\end{theorem}

\Proof
Let $V \subset M$ be a small open disk which contains $\hat{p}$ and a point of transversal homoclinic intersection associated to $\hat{p}$, such that for some $k\in \mathbb{N}$, $\Lambda:=\bigcap_{n\in \mathbb{Z}}f_1^{kn}(V)$ is an invariant hyperbolic compact set of saddle type for $f_1^k$. By choosing $V$ suitable and $k$ large enough, we may suppose that $f^k_1\mid_\Lambda$ is conjugate to a shift of $d+1$ symbols $\{0,1, \dots, d\}$. Moreover, we can assume that the elements of Markov partition of $\Lambda$ are contained in open sets such that their closure are pairwise disjoint. Also observe that if $k$ is large enough and $V$ is suitable chosen, then $d=d(f_1, k)$ can be taken arbitrarily large. 
Therefore if $k$ is large enough, by taking $f_1^k$, and $f_2^k$ instead of $f_1$ and $f_2$, we may assume that  $P_0=(\hat{p},\hat{q})$ is a fixed point of  $F_{0}=f_{1}\times f_{2}$ and $\Lambda$ is $f_1$-invariant and  $f_1$ on $\Lambda:=\bigcap_{n\in \mathbb{Z}}f_1^{n}(V)$ is conjugate to the shift of symbols $\{0,1, \dots, d\}$. Moreover, we may assume that $f_{2}$ is dominated by $f_{1}|_{\Lambda}$.  Indeed, we may replace $f_{2}$ by a map which is identity in the complement of $\hat{U}$, where we choose $\tilde{U}\subset \subset \hat{U}$ small neighborhoods of the fixed point $\hat{q}$, and in such a way that $\tilde{U}$ is much smaller than $\hat{U}$.

For a (linear) contraction map of $\RR^{n}$, ($2n= \dim(N)$) with contraction bound equal to $1/2$,
Proposition \ref{pro ifs 1} gives a number $l$ as a required number of elements of the IFS to obtain transitivity in some small disk. 

Now we choose $k$ large enough such that $d\geq2l$.

Once we substitute $f_{2}$ by $f_2^k$ it follows that $m(Df^k|_{E^s})> (1-\delta)^k$. Therefore, $\delta=\delta(f_1,\dim(N))$ is chosen close to zero in such a way that $(1-\delta)^k>1/2$.

From now on, we assume that $\hat{q}$ of $f_{2}$ is a $\delta$-weak hyperbolic fixed point of $f_{2}$ with $\delta<1/2$, $\hat{p}\in \Lambda$ is a hyperbolic fixed point of  $f_{1}$, and
$\Lambda$ is $f_1$-invariant and it is conjugate to a shift with $d+1$ symbols as above.

We want to perturb $F_{0}$ to construct a family $F_{\mu}$  such that each one has a symbolic double blender as a sub-system, and we will show how this leads to the existence of a symplectic blender.


\subsubsection*{Perturbations}
Let $\zeta>0$ is small enough (will be chosen later) and $\varepsilon : [0,1] \longrightarrow [0,\zeta]$ is a smooth simple curve such that $\varepsilon(0)=0$.

Let us define $F_0=f_1\times f_2$.\\
For $\mu \in (0,1]$,
$$F_{\mu}:= \Psi^{\varepsilon(\mu)}\circ F_0 \circ  \Phi^{-\varepsilon(\mu)},$$
where $ \Psi^t$ and $\Phi^{t}$ are the time $t$ maps of the Hamiltonians $\tilde{h}_{1}$ and $\tilde{h}_{2}$, respectively, to be chosen later.

In order to define the local perturbation, we first consider open sets $\cA_{ij}$ and pairwise disjoint open sets $\widetilde{\cA}_{ij} $ in a small neighborhood of $\Lambda$, such that,
\begin{eqnarray}\label{equ aij}
 \cA_{ij} \cap \Lambda = \left\lbrace (x_i)_{i\in \mathbb{Z}} \mid x_0= i, ~ x_1=j \right\rbrace \andrm \cA_{ij} \subset \subset  \widetilde{\cA}_{ij}
\end{eqnarray}
where $(x_i)_{i\in \mathbb{Z}}$ is the itinerary of a point in the Markov partition.

 In addition, we set 
\begin{eqnarray}\label{equ a*}
 \cA_{i,*}=\bigcup^{d}_{j=0} \cA_{ij} \andrm \cA_{*,j}=\bigcup^{d}_{i=0} \cA_{ij}.
\end{eqnarray}
 If $\cJ, \cI\subset \{0,1, \dots, d\}$ then we set 
\begin{eqnarray}\label{equ aiJ}
\cA_{i,\cJ}=\bigcup_{j\in \cJ} \cA_{ij}.
\end{eqnarray}
Similar notations shall be used for $\cA_{\cI,j}, \cA_{\cI,\cJ},   \widetilde{\cA}_{i,*}, \widetilde{\cA}_{*,j}$ and  $\widetilde{\cA}_{\cI,\cJ}$.

\begin{figure}
\begin{center}
\includegraphics{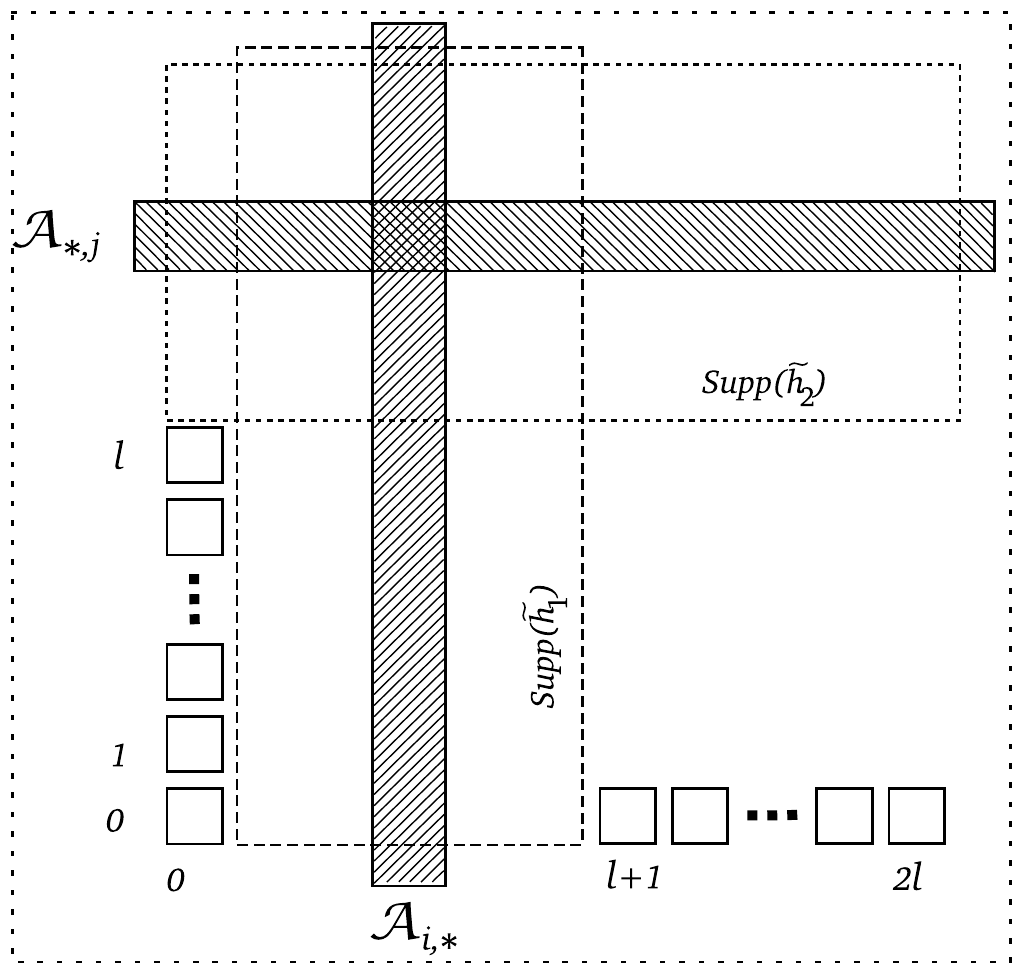}
\caption{Support of local perturbations projected to $\Lambda$. No perturbation is made in white blocks.}
\end{center}
\end{figure}

Using the Darboux Theorem we select local coordinates in a neighborhood $\tilde{U}$ of $\hat{q}$ in $N$, of the form $(a_{1}, \dots, a_{n}; b_{1}, \dots, b_{n})$ such that on $\tilde{U}$ the symplectic form is  $\Sigma_{i=1}^{n} da_{i}\wedge db_{i}$. These coordinates will be useful to define our local perturbation. 

Let $U \subset \subset \tilde{U}$ be a small neighborhood of $\hat{q}$ such that $f_{2}(U)\subset \subset \tilde{U}$.

Given a sufficiently small vector  $(u,v)=(u_{1}, \dots, u_{n}; v_{1}, \dots, v_{n})\in \RR^{2n}$, let the Hamiltonian $h_{(u,v)}$ on $N$, be a bump function  such that
$h_{(u,v)}(y)=0$ if $y\in N\setminus \tilde{U}$ and  if $y\in U$; then $h_{(u,v)}$ is expressed in the local coordinates as ($y=(a,b)\in U$)
 $h_{(u,v)}(a,b):= a\cdot v - b\cdot u$.
 
 Note that the time $t$ map of the Hamiltonian flow of $h_{(u,v)}$ in $U$ is  the translation 
 $$(a,b)\mapsto (a,b) + t(u,v).$$

We identify a neighborhood of zero in $T_{x}N$ and $N$ in the local coordinates on $\tilde{U}$, and let $\hat{q}=0$ in the local coordinates. So $Df_{2}(\hat{q})$ is very close to $f_{2}$ on the  neighborhood $U$ of $\hat{q}$, provided that $U$ is small enough.

Let  $\varphi^{s}=Df_2(\hat{q})|_{E^s_{\hat{q}}}$. By the assumption, the contraction bound of $\varphi^{s}$ is $1-\delta>1/2$. Then, from
Proposition \ref{pro ifs 1} we get small vectors $c_{1}, \dots, c_{l}$ in $E^s_{\hat{q}}\subset T_{\hat{q}}N$ such that the linear maps 
$\varphi^{s}_0:=\varphi^{s},
\varphi^{s}_1 :=\varphi^{s} + c_{1},  \dots , \varphi^{s}_l :=\varphi^{s} + c_{l}$ on $E^{s}_{\hat{q}}$ generate an iterated function system which is transitive in some small disk $D^s \subset E^s_{\hat{q}}$
(satisfying the covering and well-distributed properties).
  
We denote the vectors $c_{i}$ in the local coordinates as $(u^{i},v^{i})\in \tilde{U}$.
So we may define the Hamiltonians $h_{(u^{i},v^{i})}$ on $N$ as above.

Now, we define the Hamiltonian perturbation $\tilde{h}_{1}: M\x N \to \RR$,  as a bump function such that 
$$\tilde{h}_{1}(x,y)= h_{(u^{i},v^{i})}(y) \ifrm x\in\cA_{i,*}, i=1, \dots, l,$$
$$\tilde{h}_{1}(x,y)=0   \ifrm x\notin \bigcup^{l}_{i=1} \widetilde{\cA}_{i,*}.$$

Similarly, we get a family of contracting maps $\varphi^{u}_0:=Df^{-1}_2(\hat{q})|_{E^u_{\hat{q}}},$
 $\varphi^{u}_1 :=\varphi^{u} + c'_{1},$  $\dots , \varphi^{u}_l :=\varphi^{u} + c'_{l}$ on $E^{u}_{\hat{q}}$ that generate a transitive iterated function system in a small disk $D^u$.
We denote the vectors $c'_{i}$ in the local coordinates as $(u'^{i},v'^{i})\in \tilde{U}$.

The Hamiltonian perturbation $\tilde{h}_{2}: M\x N \to \RR$ is a bump function such that 
$$\tilde{h}_{2}(x,y)= h_{(u'^{i},v'^{i})}(y) \ifrm x\in\cA_{*,j}, j=l+1, \dots, 2l,$$
$$\tilde{h}_{2}(x,y)=0   \ifrm x\notin \bigcup^{2l}_{j=l+1} \widetilde{\cA}_{*,j}.$$


\subsubsection*{The symplectic blender}
Now we  explain how this construction  creates a symplectic blender: First, observe that by the construction we get a geometric model of a symbolic double-blender and therefore its  perturbations can be treated like the perturbations of a double-blender.  Thus the proof reduces to the argument in Section \ref{s geometric double blenders are double blender}. Since we preserve a symplectic structure, the double-blender becomes a symplectic blender.

Moreover, we have the following proposition which is a consequence of the
first part of Proposition \ref{pro ifs 2}. The details of the proof are left to the reader.
\begin{proposition} \label{pro B4}
Under the hypotheses of Theorem \ref{thm blender} it is possible to create a 
symplectic blender with the following additional property:
\begin{list}{{\rm \textbf{B}-\arabic{Lcount} $\;$} }{\usecounter{Lcount}}
\setcounter{Lcount}{3}
\item  Any forward and backward iteration of a $uu$-leaf ($ss$-leaf) intersecting
$\Lambda\x U$, intersects $\Lambda \x U$ in a $uu$-segment ($ss$-segment,
respectively).
\end{list}
Consequently, the set of all points whose strong (un)stable manifolds intersect
$\Lambda\x U$, is an  invariant set.
\end{proposition}

\subsection{Blender inside partial hyperbolic sets}\label{blender in ph}
In the present section we give a series of lemmas about the role of blenders in obtaining robust transitivity. 
Roughly speaking, Lemmas \ref{lem blender ph 1} and \ref{lem blender ph 2} state that the points in a partially hyperbolic set connected to the blender through the strong leaves are contained in a homoclinic class. These lemmas are essential in the proof of Theorem \ref{thm A'}.

\begin{definition} Given a (not necessarily compact) partially hyperbolic invariant set $\Lambda$ with splitting  $ E^{ss} \oplus E^{cs} \oplus E^{cu} \oplus E^{uu}$, for any $z\in \Lambda$ we define the strong stable (unstable) manifold of $z$ with diameter $L$ as the set 
 $W^{ss}_L(z):= f^{-L}(W^{ss}_{loc}(f^L(z)))$
 ($W^{uu}_L(z):= f^{L}(W^{ss}_{loc}(f^{-L}(z)))$).
\end{definition}

\begin{remark} \label{rmk blender ph 2} Recall that for a blender $(p, \cB)$, if $W^{ss(uu)}_L(x)\cap W^{u(s)}(B)\neq \vazio$ then for any $y$ sufficiently close to $x$  it  also holds that $W^{ss(uu)}_L(y)\cap W^{u(s)}(B)\neq \vazio$, where $B$ is the maximal invariant set  generated by $\cB.$
\end{remark}

\begin{lemma}\label{lem blender ph 1}
Let $\Lambda$ be a (not necessarily compact) partially hyperbolic invariant set with splitting  $ E^{ss} \oplus E^{cs} \oplus E^{cu} \oplus E^{uu}$.
Assume that
\begin{enumerate}
\item there exists a symplectic blender $(p,\cB)$, such that its  maximal invariant hyperbolic set $B$ is contained in $\Lambda$  and has unstable dimension equal to $\dim(E^{cu} \oplus E^{uu})$;

\item there exists a non empty set $R \subset \Lambda$ such that if $x \in R$ then
 there are infinity many forward and backward iterates of $x$ in $\cB$.

\end{enumerate}
Then, the homoclinic class of $B$ contains the closure of $R$.
\end{lemma}

\begin{proof} Let us fix $x\in R$. From the fact that there are infinity many  backward iterates of $x$ in $\cB$ and from Remark \ref{rmk blender ph 1}, it follows that there is a sequence of points accumulating on $x$ that belong to the intersection of the unstable manifold of $B$ with the strong stable manifold of $x$.
In the same way, and using now that there are infinity many forward  iterates of $x$ in $\cB$ it follows that there is a sequence of points accumulating on $x$ that belong to the intersection of the stable manifold of $B$ with the strong unstable manifold of $x$. Using Remark  \ref{rmk blender ph 2} we conclude that there is a sequence of points accumulating on $x$ that belongs to the intersection of the stable and unstable manifold of $B.$
\end{proof}

\begin{lemma}\label{lem blender ph 2}
Let $\Lambda$ be a (not necessarily compact) partially hyperbolic invariant set with splitting  $ E^{ss} \oplus E^{cs} \oplus E^{cu} \oplus E^{uu}$.
Assume that
\begin{enumerate}
\item there exists a symplectic blender $(p,\cB)$, such that its  maximal invariant hyperbolic set $B$ is contained in $\Lambda$  and has unstable dimension equal to $\dim(E^{cu} \oplus E^{uu})$;

\item there exists a non empty set $R \subset \Lambda$ contained in the recurrent set of $f$ such that   if $x \in R$ then
\begin{itemize}
\item[(2.1)] $W^{ss}(x)\cap W^u_{loc}(B)\neq \vazio$,
\item[(2.2)] $W^{uu}(x)\cap W^s_{loc}(B)\neq  \vazio$.
\end{itemize}

\end{enumerate}
Then, the homoclinic class of the $B$ contains the closure of $R$.
\end{lemma}

\begin{proof} If $W^{ss}(x)\cap W^u_{loc}(B)\neq  \vazio$, there is $L$ sufficiently large such that
 $W^{ss}_L(x)\cap W^u_{loc}(B)\neq  \vazio$. From Remark \ref{rmk blender ph 2} and the fact that $x$ is recurrent, it follows that there is a sequence
$\{f^{-n_k}(x)\}$ such that $ W^{ss}_L(f^{-n_k}(x))\cap W^u(B)\neq  \vazio$ and therefore, there is a sequence of points accumulating on $x$ that belong to the intersection of the unstable manifold of $B$ with the strong stable manifold of $x$.  Arguing on the same way, it follows that there is a sequence of points accumulating on $x$ that belong to the intersection of the stable manifold of $B$ with the strong unstable manifold of $x$.
The lemma is now obtained as in the proof of Lemma \ref{lem blender ph 1}.
\end{proof}


\section{Parametric version of Theorem \ref{thm A}} \label{sec proof A}


In this section,  we prove a parametric version of  Theorem \ref{thm A} in a more general context. The proof of our main theorems are based on this theorem and its proof.

\begin{theorem} \label{thm A'}
Let $M$ and $N$ be two symplectic manifolds (not necessarily compact), and $1\leq r\leq \infty$. Let $f_{1}\in \mathrm{Diff}^{r}_{\omega}(M)$ such that there exists an open set $V\subset M$ which maximal invariant set $\Lambda$ is a nontrivial topologically mixing hyperbolic  compact set. Let  $f_{2}\in \mathrm{Diff}^{r}_{\omega}(N)$ such that: 
\begin{list}{{\rm (\alph{Lcount})}}{\usecounter{Lcount}}
\item $f_2$ is dominated by $f_1\vert_{\Lambda}$, and
$f_2$ has a $\delta$-weak hyperbolic periodic point for some
positive $\delta=\delta(f_1,\dim(N))$ close to zero. 
\item For any $\tilde{f}_2$ sufficiently $C^r$ close to $f_2$,
$\Omega(\tilde{f}_2)=N$.
\end{list}

\noindent
Then there is a $C^{r}$ arc $\{F_{\mu}\}_{\mu\in [0,1]}$ of $C^{r}$ symplectic
diffeomorphisms on $M\x N$, such that
$F_{0}=f_{1}\x f_{2}$, and there exists a periodic orbit $P$ such that  
for all $\mu\in(0,1]$ the following hold, 
\begin{enumerate}
\item $P$ is periodic for $F_\mu$, $H(P;F_\mu)\bigcap_{n\in\ZZ} F_\mu^n(V\x N) = \Lambda \x N$ and it is robustly strictly  topologically mixing (See Definition \ref{def RT}).
\item If $G$ is close to $F_\mu$, the set  $\Gamma_G:=H(P_G;G)\bigcap_{n\in\ZZ} G^n(V\x N)$ is robustly strictly  topologically mixing.
\item For any compact neighborhood $N_c\subset N$ the map $G\in \cU \mapsto \Gamma_G \bigcap (V\x N_c)$ defines a continuation of $\Lambda\x N_c$, for some neighborhood  $\cU$ of  $F_\mu$. 

\end{enumerate}
\end{theorem}

\begin{remark}\label{rem thm A' bis}

We want to highlight that  in Theorem \ref{thm A'} it is not assumed that the manifolds $M$ and $N$ are compact. However the hypotheses of Theorem \ref{thm A'} are contained in the hypotheses of Theorems \ref{thm A}, \ref{thm B}, \ref{thm C} and \ref{thm D}. In particular, observe that in Theorem \ref{thm A'} we assume that  for any $\tilde{f}_2$ sufficiently $C^r$ close to $f_2$ it holds that $\Omega(\tilde{f}_2)=N$, which is immediately satisfied whenever either $ M$ is compact or has finite volume.

\end{remark}

\begin{remark}\label{rem thm A'}

First observe that for any diffeomorphism  $F_\mu$, with $\mu>0$,  it holds that the whole set $\La\x N$ is transitive. In this sense, those maps satisfy a stronger property than the ones in the thesis of Theorem \ref{thm A}. In fact, the ones in the thesis of Theorem \ref{thm A} are obtained as perturbations of the family map $\{F_{\mu}\}_{\mu\in (0,1]}.$
Also observe, that the   non-wandering hypothesis (b) is obviously satisfied if the manifold $N$ is compact or has a finite volume.

\end{remark}

Before giving the proof of Theorem \ref{thm A'} we need some results stated in Section \ref{s hps}, about   the persistence of normally hyperbolic laminations.  This  gives us the right  framework to obtain the continuation of the set in the thesis of Theorem \ref{thm A'}. The proof of Theorem \ref{thm A'} is postponed to Section \ref{proof thm A'}.


\subsection{Persistence of normally hyperbolic laminations} \label{s hps}
Now, we recall a technical version of the main results of Hirsch-Pugh-Shub \cite{hps} on persistence of
normally hyperbolic laminations, extended to the non-compact embedded case. As we said before, the result provides the existence  and precise characterization of the continuation of a partially hyperbolic sets (for the definition of continuation used in the present context, see Section \ref{def strong continuation}).

\begin{definition}[strong continuation]\label{def strong continuation}
An invariant set $X \subset M$ of $f$ has {\it strong continuation} in $\cD^r$, if there exist an open neighborhood $\cU$ of $f$ in $\cD^r$, and a continuous map $\Phi:\cU \to
\mathcal{P}(M)$ such that, $\Phi(f)=X$, and for any $g\in \cU$, the set
$\Phi(g)\subset M$ is homeomorphic to $X$ and  invariant for $g$. 
Then, $\Phi(g)$ is called the \textit{strong  continuation} of $X$ for $g$.
Here, $\mathcal{P}(M)$ is the space of all subsets of $M$ with the Hausdorff
topology.
\end{definition} 
Compare this definition with Definition \ref{def continuation}. In the present one, it is required that $\Phi(g)$ is homeomorphic to $X$. 

\begin{theorem}[{\bf [HPS]}]  \label{thm hps}
Let $M$ and $N$ be two boundaryless manifolds (not necessarily compact). Let $1\leq r\leq \infty$ and  let $f_{1}\in \mathrm{Diff}^{r}_{\omega}(M)$ such that there exists an open set $V\subset M$ which maximal invariant set $\Lambda$ is a nontrivial topologically mixing hyperbolic  compact set. 
Let  $f_{2}\in \mathrm{Diff}^r(N)$ be
dominated by $f_1\vert_{\Lambda}$.
Then the invariant set  $\Lambda \x N$ of  $F_0 = f_1\x f_2$ has a unique strong continuation for any small perturbation of $F_0$
in $\mathrm{Diff}^r(M\x N)$.
\end{theorem}
More precisely, the following hold,
\begin{list}{\textbf{H}-\arabic{Lcount} $\;$}{\usecounter{Lcount}}

\item There is a neighborhood $\mathcal{U}\subset$ Diff$^{1}(M\times N)$ of $F$ such that every $G\in \mathcal{U}$ has a (locally maximal) invariant $\Gamma_{G}$ homeomorphic to $\Lambda\times N$ and is a continuation of  $\Gamma_{F_0}$. 

\item  There is a $G$-invariant lamination on $\Gamma_{G}$ by manifolds diffeomorphic to $N$.
So $G$ induces a homeomorphism $\tilde{G}$ on the quotient of $\Gamma_{G}$ such that  $\tilde{G}$ is conjugate to $f_1|_{\Lambda}$.

\item From previous item, $G$ restricted to $\Gamma_{G}$ is conjugate to a skew product $G^{*} :~ (x,w)\longmapsto (f_1(x), g_{x}(w))$ on $\Lambda\times N$, which depends continuously on $G$.
\end{list}

This result, in a more general setting,  is proved in \cite{hps}  in the compact case. However, as remarked there, the compactness is used only  to obtain uniform estimates on the functions involved in the proof (i.e. the rate  of contraction and expansion)  and the results carry over as long as such uniform estimates hold. This is the case in our context since we consider the uniform $C^r$ topology.
See \cite{bw}, \cite[Appendix B]{dls06} and \cite[Appendix A]{dgls} in which the similar situation is treated.

\subsection{Proof of Theorem \ref{thm A'}}
\label{proof thm A'}  
The main idea of the proof is to find perturbations of $F_0=f_1 \x f_2 $ so that the following hold:

(i) the existence of a symplectic blender,

(ii)  the ``minimality'' of (strong) stable and unstable foliations in the partially hyperbolic set $\Lambda \x N$.

These properties imply the transitivity (or even topological mixing) of the set  $\Lambda \x N$ in a robust fashion.
Following this approach, the proof  is constructive and it is divided in four parts. 
First, in Section \ref{s pert thm A} we introduce the perturbations which are  similar to the ones used in the construction of the family $F_{\mu}$ in the proof of Theorem \ref{thm blender}. The perturbations are done using the result stated in Section \ref{s hps}.
Then in the  sequel subsections we prove that the perturbed systems satisfy the desired properties. 
In Section \ref{s blender thm A}, applying Theorem \ref{thm blender} we show the existence of a symplectic blender. 
Then in Section \ref{s min thm A} we use the results of iterated function systems of recurrent diffeomorphisms (see Section  \ref{s ifs rec}) to prove that the strong stable and unstable manifolds of almost all points in the central manifold intersect the constructed blender. 
In Section \ref{s robust thm A}, using the lemmas in Section \ref{blender in ph}, we show that this property is robust under small perturbations.

\Proofof{Theorem \ref{thm A'}} Let $\hat{p}$ be a hyperbolic periodic point of $f_1$ with transversal homoclinic intersection, and  $\hat{q}$ be a $\delta$-weak hyperbolic periodic point of $f_2$. Let $P_0=(\hat{p}, \hat{q})$.
Let $V \subset M$ be a small open disk which contains $\hat{p}$ and a transversal homoclinic point associated to it, such that for some $k\in \mathbb{N}$, $\Lambda_0:=\bigcap_{n\in \mathbb{Z}}f_1^{kn}(V)$ is an invariant hyperbolic compact set of saddle type for $f_1^k$. 
By choosing $V$ suitable and $k$ large enough, we may suppose that $f^k_1\mid_{\Lambda_0}$ is conjugate to a shift of $d+1$ symbols $\{0,1, \dots, d\}$. Moreover, we can assume that the elements of Markov partition of $\Lambda_0$ are contained in open sets such that their closure are pairwise disjoint. If $k$ is large then $d$ can be taken large. 
If $k$ is large enough, by taking $f_1^k$, and $f_2^k$ instead of $f_1$ and $f_2$, we may assume that  $P_0=(\hat{p}, \hat{q})$ is a fixed point of  $F_{0}=f_{1}\times f_{2}$ and $\Lambda_0\subset \Lambda$ are $f_1$-invariant. Moreover on $\Lambda_0:=\bigcap_{n\in \mathbb{Z}}f_1^{n}(V),$ $f_1$ is conjugate to the shift of symbols $\{0,1, \dots, d\}$.

As in the proof of Theorem \ref{thm blender} recall that for a (linear) contraction map of $\RR^{n}$, ($2n= \dim(N)$) with contraction bound equal to $1/2$, 
Proposition \ref{pro ifs 1} 
gives a number $l$ as the required number of elements of the IFS needed to obtain transitivity in some small disk.

We choose $k$ large enough so that $d\geq2l+4$ (observe that the number $d$ here is larger than the on in Theorem \ref{thm blender}).

Again, as in the proof of Theorem \ref{thm blender}, once we substitute $f_{2}$ by $f_2^k$ it follows that $m(Df^k|_{E^s})> (1-\delta)^k$. Therefore, the constant  $\delta=\delta(f_1,\dim(N))$ is chosen close to zero in such a way that $(1-\delta)^k>1/2$.

From now on, we assume that $\hat{q}$  is a $\delta$-weak hyperbolic fixed point of $f_{2}$ with $\delta<1/2$, $\hat{p}\in \Lambda_0$ is a hyperbolic fixed point of  $f_{1}$, and
$\Lambda_0$ is $f_1$-invariant and it is conjugate to a shift with $d+1$ symbols as above.

In what follows, we are going to focus on the set $\Lambda_0\x N.$


\subsubsection{The perturbations}\label{s pert thm A}

Let $\zeta>0$ is small enough (will be chosen later) and $\varepsilon : [0,1] \longrightarrow [0,\zeta]$ such that  $\varepsilon : x\mapsto x/\zeta$.

For $\mu \in [0,1]$,
$$F_{\mu}:= \Psi^{\varepsilon(\mu)}\circ F_0 \circ  \Phi^{-\varepsilon(\mu)},$$
where $ \Psi^t$ and $\Phi^{t}$ are the time $t$ map of the Hamiltonians $h_{1}$ and $h_{2}$, respectively, to be chosen later. It is clear that $F_0=f_1\times f_2$.

As in Theorem \ref{thm blender} using the Darboux Theorem we select local coordinates in a neighborhood $\tilde{U}$ of $\hat{q}$ in $N$, of the form $(a_{1}, \dots, a_{n}; b_{1}, \dots, b_{n})$ such that on $\tilde{U}$ the symplectic form is  $\Sigma_{i=1}^{n} da_{i}\wedge db_{i}$. This coordinates will be useful to define our local perturbation. 

Let $U\subset \subset \tilde{U}$ be a small neighborhood of $\hat{q}$ such that $f_{2}(U)\subset \subset \tilde{U}$.

As in the proof of Theorem \ref{thm blender}, we set  
$$\cA_{i,*}=\bigcup^{d}_{j=0} \cA_{ij} \andrm \cA_{i,\cJ}=\bigcup_{j\in \cJ} \cA_{ij}$$
 where $\cJ \subset \{0,1, \dots, d\}$. Similar notations shall be used for $\cA_{*,j},$ $\cA_{\cI,j}, \cA_{\cI,\cJ},$   $\widetilde{\cA}_{i,*},$  $\widetilde{\cA}_{*,j}$ and  $\widetilde{\cA}_{\cI,\cJ}$.

Now, we separate the support of perturbations  in two groups: 
\begin{enumerate}
 \item one composed by $\cA_{i,*},i=1, \dots, l,$ and $\cA_{*,j}, j=l+1, \dots, 2l;$
\item  the other composed by $\cA_{\cJ_1,\cJ_1}\cup\cA_{\cJ_2,\cJ_2}$, where  $\cJ_1=\{0, 2l+1, 2l+2\}$ and $\cJ_2=\{0, 2l+3, 2l+4\}.$
\end{enumerate}
In the first group of sets, the same perturbations done in the proof of Theorem \ref{thm blender}  is performed along the center direction. This shall provide us a blender (see Section \ref{s blender thm A}).  In the second ones, other perturbations (explained below) are introduced such that they provide the minimality of the unstable and stable foliations (see Section \ref{s min thm A}). 

So, using the same definitions of the proof of Theorem \ref{thm blender}  we define the Hamiltonian perturbation $\tilde{h}_{1}: M\x N \to \RR$,  as a bump function such that 
$$\tilde{h}_{1}(x,y)= h_{(u^{i},v^{i})}(y) \ifrm x\in\cA_{i,*}, i=1, \dots, l,$$
$$\tilde{h}_{1}(x,y)=0   \ifrm x\notin \bigcup^{l}_{i=1} \widetilde{\cA}_{i,*}.$$

Similarly,  the Hamiltonian perturbation $\tilde{h}_{2}: M\x N \to \RR$ is a bump function such that 
$$\tilde{h}_{2}(x,y)= h_{(u'^{i},v'^{i})}(y) \ifrm x\in\cA_{*,j}, j=l+1, \dots, 2l,$$
$$\tilde{h}_{2}(x,y)=0   \ifrm x\notin \bigcup^{2l}_{j=l+1} \widetilde{\cA}_{*,j}.$$

Now we explain the perturbation performed in the set $\cA_{\cJ_1,\cJ_1}\cup\cA_{\cJ_2,\cJ_2}$ ($\cJ_1=\{0, 2l+1, 2l+2\}, \, \cJ_2=\{0, 2l+3, 2l+4\}).$

Let $\hat{h}_1$ and $\hat{h}_2$ be two integrable Hamiltonians  on $N$, close to the identity, such that the time one maps of their corresponding Hamiltonian flows satisfy the properties (1) and (2) in Lemma \ref{lem tori}.

We define the Hamiltonian  $\tilde{h}_{3}: M\x N \to \RR$ as a bump function such that 
$$\tilde{h}_{3}(x,y)= \hat{h}_{1}(y) \ifrm x\in\cA_{2l+1,j} , ~ j\in\{0,2l+1, 2l+2\},$$
$$\tilde{h}_{3}(x,y)= \hat{h}_2(y) \ifrm x\in\cA_{2l+2,j},  j\in\{0,2l+1, 2l+2\}, $$
$$\tilde{h}_{3}=0   \ifrm x\notin  \widetilde{\cA}_{i,j},   i\in\{2l+1,2l+2\}, j\in \{0,2l+1, 2l+2\}.$$

Similarly, we define the Hamiltonian  $\tilde{h}_{4}: M\x N \to \RR$ as a bump function such that 
$$\tilde{h}_{4}(x,y)= -\hat{h}_{1}(y) \ifrm x\in\cA_{i,2l+3} , ~ i\in\{0,2l+3, 2l+4\},$$
$$\tilde{h}_{4}(x,y)= -\hat{h}_2(y) \ifrm x\in\cA_{i,2l+4}, ~ i\in\{0,2l+3, 2l+4\}, $$
$$\tilde{h}_{4}=0   \ifrm x\notin  \widetilde{\cA}_{i,j}, ~  i\in\{0,2l+3, 2l+4\}, j\in \{2l+3, 2l+4\}.$$

Observe that the support of these Hamiltonians are disjoint from the supports of $\tilde{h}_{1}$ and $\tilde{h}_{2}$. 
Now we are able to define the Hamiltonians  $h_{1}$ and $h_{2}$. 
Let $\epsilon>0$ small enough, we set
$$h_{1}= \tilde{h}_1 +\epsilon\tilde{h}_3,$$
$$h_{2}= \tilde{h}_2 +\epsilon\tilde{h}_4.$$
In the next two sections, we show how the above perturbation maps exhibit symplectic blenders and ``almost'' minimality of the strong foliations.

\begin{figure}[]
\begin{center}
\includegraphics[width=10cm]{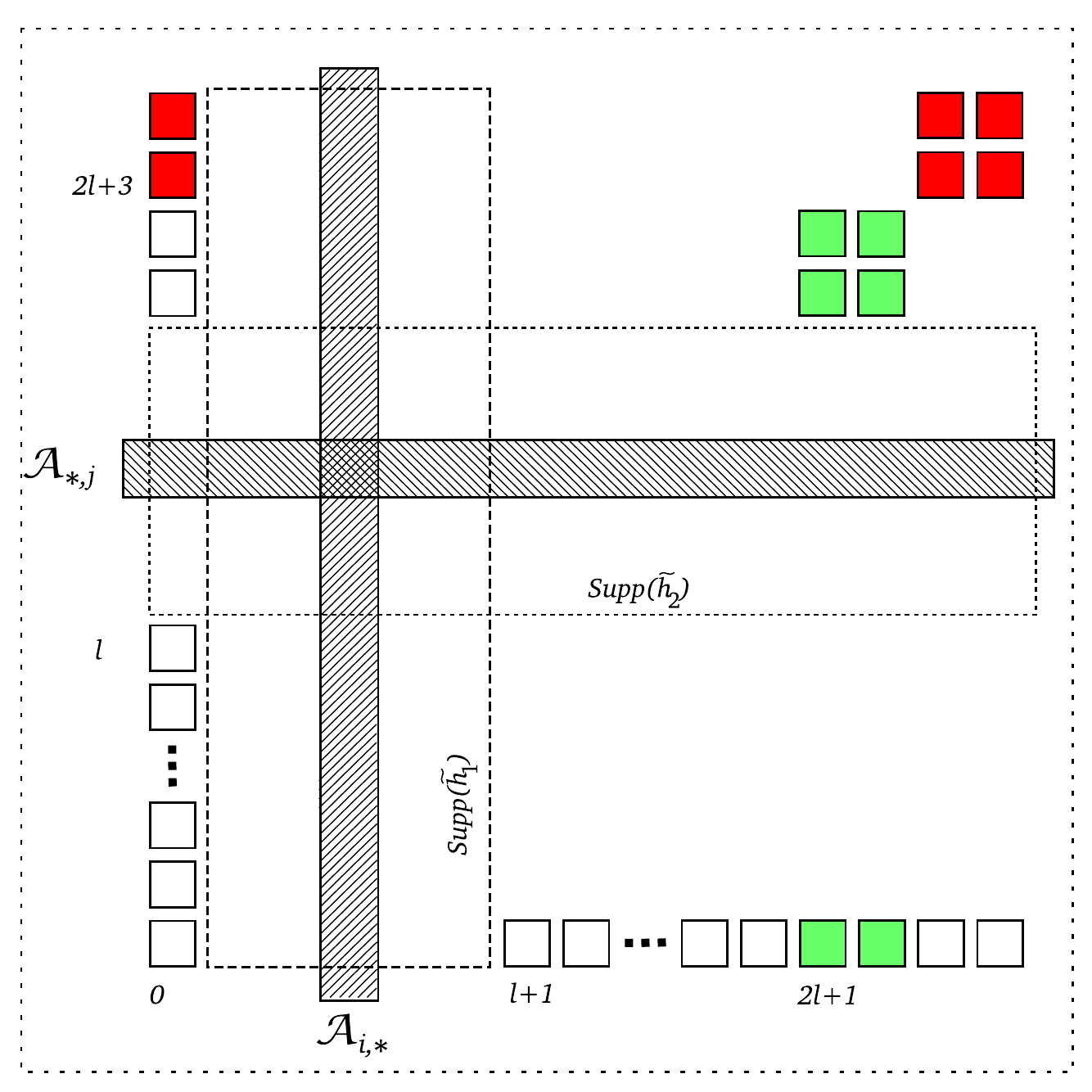}
\label{fig thm A}
\caption{Support of local perturbations projected to $\Lambda_0$. The blocks with the same color are in the support of the same Hamiltonians. No perturbation is made in white blocks.}
\end{center}
\end{figure}


\subsubsection{The symplectic blender}\label{s blender thm A}
We may repeat the proof of Theorem \ref{thm blender} for the family $F_{\mu}$ constructed in Section \ref{s pert thm A} to obtain a symplectic blender $\cB$. In addition, Proposition \ref{pro B4} also holds  for this family.


\subsubsection{Almost minimality of stable and unstable foliations} \label{s min thm A}
In this section it is shown that the strong stable and unstable manifolds of an open and dense set of points in the central manifold $N_{_0}:= \{\hat{p}\}\x N$ intersects the constructed blender. From that, we obtain the existence of an open and dense set of points in $\Lambda\x N$ such that their strong (stable and unstable) manifolds are connected to the blender. We refer to this property by the almost minimality of the strong stable and strong unstable foliations.

\begin{proposition}\label{proof minimal}
Let us fix  $\mu>0$. Then there is an open and dense set $\cR_\mu \subset N$  such that for every $q \in \cR_\mu $
and for any $n\in \ZZ$ it follows that 
$W^{uu}(F^n_{\mu}(\hat{p},q)) \cap \cB \neq \varnothing$ and $W^{ss}(F^n_{\mu}(\hat{p},q)) \cap \cB \neq \varnothing$.
\end{proposition}

\Proof 
The key elements in the proof are the symbolic dynamics, the results of Section
\ref{s ifs rec} and Proposition \ref{pro B4}.

Here again we consider restriction of $f_1$ to $\Lambda_0$. For any $p=(p_{i})_{i\in\mathbb{Z}}$ $ \in \Lambda_0=\{0,1, 2,$  $\dots, d\}^{\mathbb{Z}}$, the local and global unstable manifolds of $p$ for $f_1$ are
$$W^{u}_{loc}(p ~;~ f_1|_{\Lambda_0})=\{(x_{i})\mid \forall n\leq 0, x_{i}=p_{i} \},$$
$$W^{u}(p~;~f_1|_{\Lambda_0})=\{(x_{i})\mid \exists n_{0}\in\mathbb{Z}, \forall n\leq n_{0}, x_{i}=p_{i} \}.$$
So, the local and global strong unstable manifolds of $(p,q)$ for $F_{\mu}$ are
$$W^{uu}_{loc}(p,q;F_{\mu}|_\Gamma)=W^{u}_{loc}(p;f_1|_{\Lambda_0})\times \{q\}=\{(x_{i})\mid \forall n\leq 0, x_{i}=p_{i} \}\times \{q\},$$
$$W^{uu}(p,q;F_{\mu}|_\Gamma)=\bigcup_{n\geq0} F^{n}_{\mu}(W^{uu}_{loc}(F^{-n}_{\mu}(p,q);F_{\mu}|_\Gamma)).$$

Let $T_1=f_2$, $T_2=\phi_1\circ f_2 $ and $T_3=\phi_2\circ f_2$, where $\phi_1$ and $\phi_2$ are the time $t=\veps(\mu)$ maps of the Hamiltonian flows of  $\epsilon\hat{h}_{1}$ and $\epsilon\hat{h}_{2}$, respectively.
Then the proof of Theorem \ref{thm rec} yields the transitivity of the IFS $\cG(T_{1}, T_{2}, T_{3})$.

Let $q\in {\rm Rec}(f_2) \subset N$ such that there is a finite sequence $(\sigma_{i})_{i=1}^{n}$  ($\sigma_i \in \{1,2,3\}$) and 
$$T_{\sigma_n}\circ T_{\sigma_{n-1}} \circ \cdots \circ T_{\sigma_1} (q) \in T_1^{-2}(D). $$
We denote the set of all such points by $R_\mu$. 

Now, we consider
$$x=(x_{i})=(\overbrace{\dots, 0 ,0,0 }^{W_{loc}^u(\hat{p})} ~; ~ 
\overbrace{a_{1} , a_{2}, \dots, a_{n_0}}^{IFS}, 0, 0,\overbrace{x_{n_0 + 3},\dots}^{arbitrary}),$$
where for $i=1, 2, \dots, n_0$,
\[ a_i =
\begin{cases} 
0,  &\text{if  $\sigma_i = 1$;}\\ 
2l+1, &\text{if $\sigma_i=2$;}\\
 
2l+2, &\text{if $\sigma_i=3$.} \\
\end{cases} \]

The elements of the sequence $(x_{i})$ are chosen in such a way  that $f_{1}^{n_{0}+1}(x)\in \cA_{0,0}$, $x \in W^u_{loc}(\hat{p}, f_1|_{\Lambda_0})$, and so $(x,q)\in W^{uu}(\hat{p},q;F_{\mu}|_\Gamma) $. Moreover, the choice of $a_{1} , a_{2}, \dots, a_{n_0}$ permits us to realize the IFS $\cG(T_{1}, T_{2}, T_{3})$.

We now take the iterations of the point $(x,q)$ under $F_\mu$. 
Since $F_\mu $ restricted to $\cA_{a_i , \cJ_1}\x N$ ($\cJ_1=\{0, 2l+1, 2l+2\}$) is equal to $f_1 \x T_{\sigma_i} $, inductively we have:
$$(f_1^i(x), T_{\sigma_i}\circ T_{\sigma_{i-1}} \circ \cdots \circ T_{\sigma_1} (q)) 
= F^i_\mu(x,q)
\in W^{uu}(F^i_\mu(\hat{p},q); F_\mu ). $$
In particular, for $i=n_0+1$,  $F^{n_0+1}_\mu(x,q) \in \cB$. So,
$$W^{uu}(F^{n_0+1}_\mu(\hat{p},q); F_\mu ) \cap \cB \neq \vazio. $$
 Now we apply  Proposition \ref{pro B4} which implies that for all $n\in \ZZ$,
$$ W^{uu}(F^{n}_\mu(\hat{p},q); F_\mu ) \cap \cB \neq \vazio .    $$
Let $\cR_\mu$ be the set all points $q \in N $ such that the above intersection holds. Observe that we just proved that $R_\mu \subset \cR_\mu$.
The set $\cR_\mu$ is open, because $\cB$ is open and the strong stable and unstable manifolds depend continuously on the points. 
This completes the proof.
\Endproof

\begin{remark}\label{rmk Leb}
Observe that the set  $R_\mu$
has total Lebesgue measure. This follows from the results proved in Section \ref{s ifs rec}.
\end{remark}

\begin{corollary}\label{cor proof minimal} Let $\mu>0$. Then there is an open and dense set $\hat\cR_\mu \subset \Lambda\x N$  such that for every $x \in \hat\cR_\mu $
and for any $n\in \ZZ$ it follows that 
$W^{uu}(F^n_{\mu}(x)) \cap \cB \neq \varnothing$ and $W^{ss}(F^n_{\mu}(x)) \cap \cB \neq \varnothing$.
\end{corollary}

\Proof 
First observe that Theorem \ref{thm hps} and the minimality of the stable and unstable foliations in the hyperbolic set $\Lambda$ implies that 
$W^{ss}(\{\hat{p}\}\x N ;F_\mu)$  and $W^{uu}(\{\hat{p}\}\x N ;F_\mu)$ are dense in $\Lambda \x N$.
On the other hand, the strong unstable foliations (with leaves $W^{uu}(\hat{p},q ;F_\mu)$) inside $W^{uu}(\{\hat{p}\}\x N ;F_\mu)$  
is continuous. 
Moreover, it follows from Proposition \ref{proof minimal} that a dense open subset of $uu$-leaves in $W^{uu}(\{\hat{p}\}\x N ;F_\mu)$ intersects $\cB$.
The same statements hold for the strong stable foliation.

Now, the density of $W^{uu}(\{\hat{p}\}\x N ;F_\mu)$ and $W^{ss}(\{\hat{p}\}\x N ;F_\mu)$ in $\Lambda \x N$ and Proposition \ref{pro B4} imply the corollary.
\Endproof

Observe that Corollary \ref{cor proof minimal} and Lemma \ref{lem blender ph 1} imply that $H(P;F_\mu)\bigcap_{n\in\ZZ} F_\mu^n(V\x N) = \Lambda \x N$, for any $\mu>0$.  So it remains to show that the set $H(P;F_\mu) \bigcap_{n\in\ZZ} F_\mu^n(V\x N)$ is robustly topologically mixing.


\subsubsection{Robustness of $H(P;F_\mu)\bigcap_{n\in\ZZ} F_\mu^n(V\x N)$}\label{s robust thm A}

By the definition, we have to construct an increasing sequence of compact invariant robust transitive sets, such that its union gives $\Lambda\x N.$
For any positive integer $L$ large enough, we consider the set $\hat\cR_{\mu,L}$ given by the points in $\hat\cR$ such that   for any iterate, it holds that the strong stable and strong unstable manifolds of diameter $L$ are connected to the blender.  Let us denote this set by $\Gamma_{F_\mu,L}$. It follows immediately  that the set is compact invariant and is contained in the homoclinic class of $p$. By the definition it also holds that $\bigcup_{L>0} \hat\cR_{\mu,L}=\hat\cR.$ In the same way, and following the details of the proof of Lemma \ref{lem blender ph 2} it holds that $\bigcup_{L>0} \Gamma_{F_\mu,L}=H(P;F_\mu)\bigcap_{n\in\ZZ} F_\mu^n(V\x N)$. To show that $\Gamma_{F_\mu,L}$ is robustly transitive for any $L$, observe that for any $G$ nearby to $F_\mu$ it holds that $\Gamma_{G,L}$, defined as the set of trajectories  with strong stable and strong unstable  of diameter $L$  connected to the continuation of the blender for $G$, is in fact a continuation of $\Gamma_{F_\mu,L}.$ 
 
 
\subsubsection{Continuation of compact parts}
 
Following the argument described in the previous subsection, to conclude the last item of Theorem \ref{thm A'} it is enough to show that for any compact subset  $\cK \subset \cR_\mu$, there exists $L>0$  such that any point in  $\cK$ is connected to the blender through the strong (stable and unstable) manifolds with diameter less than $L$. 
In fact, the invariant  compact parts are given by $\Gamma_{\mu,L}\bigcap_{n\in \ZZ}F^n_\mu(V\x N).$ 

 The proof of Theorem \ref{thm A'} is completed. 
\Endproof


\section{Proof of main theorems} \label{sec instability}


\subsection{Proof of Theorem \ref{thm A}}\label{s end proof A}
Here we complete the proof of Theorem \ref{thm A}.

\Proofof{Theorem \ref{thm A}} 
Observe that if $F= f_1 \x f_2$ satisfies that the hypothesis of   Theorem \ref{thm A} then it satisfies the hypothesis of   Theorem \ref{thm A'}. So, from Theorem \ref{thm A'} there exists a family $F_\mu, \mu\geq 0$  such that $F_0=f_1\x f_2$. Moreover, for any $\mu>0$, there exists a neighborhood $\cU_\mu$ of $F_\mu$ such that any $G\in \cU_\mu$ satisfies the (1)-(3) of Theorem \ref{thm A'} and so (2) of Theorem \ref{thm A}. Let $\cU=\bigcup_{\mu>0} \cU_\mu$. Then $\cU$ verifies (1) and any $G\in\cU$ verifies (2). The proof of Theorem \ref{thm A} is completed.
\Endproof


\subsection{Theorem \ref{thm B}: Robust transitivity in nearly integrable diffeomorphisms} 
The proof is similar to the proof of Theorem \ref{thm A}. However, using strongly the fact that the map $f_2$ is integrable, it is obtained that {\it no exceptional set appears.}  More precisely, the stronger conclusion follows from applying Theorem \ref{thm ifs min} instead of Theorem \ref{thm rec}.

\Proofof{Theorem \ref{thm B}} 
By the assumption, $f_1$ satisfies the {\textsf{H.S.}} hypothesis. So there exists a neighborhood $V\subset M$ such that its maximal invariant set $\Lambda$ is a transitive hyperbolic basic set.
Let $\tilde{f}_2$ be an integrable diffeomorphism sufficiently close to $f_2$  with a 
$\delta$-weak hyperbolic saddle point (with small $\delta>0$).
Then, the diffeomorphisms $f_1, \tilde{f}_2$ satisfy the hypothesis of Theorem \ref{thm A'}. Here observe that  we are  assuming that $\tilde{f}_2$ is integrable. 

So, in order to prove Theorem \ref{thm B} we adapt the proof of Theorem \ref{thm A'} (and  Theorem \ref{thm A}). In fact, we do repeat the proof of  Theorem \ref{thm A'} word by word, except in a few points to get a slightly stronger conclusion.
In the following we point out where the proof of Theorem \ref{thm A'} needs changes. Beyond that part,  the rest of the proof of Theorem \ref{thm A'} is repeated without any change.
 
First, we choose an iterate $k$ large enough such that ${f_1^k}{|\La}$ is conjugate to a shift with $d$ symbols and  $d\geq 2l + 2\dim(N)+ 2$.

Second, in the definition of perturbations (c.f Section \ref{s pert thm A}), we modify the second group of perturbations in order to apply the  Theorem \ref{thm ifs min}  instead of Theorem \ref{thm rec} in the proof of Proposition \ref{proof minimal}.

Since  $T_0:=\tilde{f}_2$ is integrable, it follows from Theorem \ref{thm ifs min} that there are $T_1, \dots, T_m$, $m:=\dim(N)+1$, arbitrarily close to $\tilde{f}_2$ such that the IFS $\cG(T_0, T_1, \dots, T_m)$ is minimal. 
Similarly, we define $\hat{T}_1, \dots, \hat{T}_m$ such that the IFS $\cG(T^{-1}_0, \hat{T}^{-1}_1, \dots, \hat{T}^{-1}_m)$ is minimal.

We set $\cJ_1=\{0, 2l+1, 2l+2, \dots,  2l+m\}$ and 
$\cJ_2=\{0, 2l+m+1, \dots, 2l+2m\}.$  

Then, we define the perturbation  performed in the set $\cA_{\cJ_1,\cJ_1}\cup\cA_{\cJ_2,\cJ_2}$ such that 
$F_\mu$ restricted on $\cA_{2l+j, \cJ_1} \x N$ is  equal to ${f}_1\x T_{j}$, for any $j\in \{1, \dots, m\}$; and
$F_\mu$ restricted on $\cA_{2l+m+j, \cJ_1} \x N$ is  equal to ${f}_1\x \hat{T}_{j}$, for any $j\in \{1, \dots, m\}$.

Now, by adapting the proof of Proposition \ref{proof minimal} we apply the minimality of the IFS $\cG(T_0, T_1, \dots, T_m)$ to prove the slightly stronger conclusion that $\cR_\mu=N$.  Consequently, in Corollary \ref{cor proof minimal} we obtain $\hat\cR_\mu = \Lambda\x N$ which is a compact set. On the other hand,  the sets $\hat\cR_{\mu,L}$ are open in $\hat\cR_\mu$. From $\bigcup_{L>0} \hat\cR_{\mu,L}=\hat\cR$ and compactness, it follows that for some large $L_0>0$, $\hat\cR_{\mu,L_0}=\hat\cR_\mu=\Lambda\x N$. Therefore, $\Gamma_{F_\mu,L_0}=H(P;F_\mu)\bigcap_{n\in\ZZ} F_\mu^n(V\x N)$.
Thus, for any $G$ close to $F_{\mu}$, the strong continuation $\Gamma_G=\Gamma_{G, L_{0}}$ is well defined and topologically mixing (from Theorem \ref{thm A'}), and it is homeomorphic to $\Gamma$.

In other word, we obtain the family of diffeomorphisms $F_\mu$  that $F_0={f}_1\x \tilde{f}_2$, and  for $\mu>0 $ the set $\Gamma=\Lambda \x N$ is robustly topologically mixing. This implies that there is a neighborhood $\cU_{\tilde{f}_2} $ of the arc  $\{F_\mu | \mu>0\}$ such that for any  $G\in \cU_{\tilde{f}_2}$ the (strong) continuation $\Gamma_G$ of $\Gamma$ is robustly topologically mixing and is a relative homoclinic class. It is clear that the projection of $\Gamma_G$ to $N$ is onto (it is  homeomorphic to  $\Lambda \x N$).  
Now let $$\cU= \bigcup_{{\tilde{f}_2} } \cU_{\tilde{f}_2},$$
where the union is taken over all integrable diffeomorphism $\tilde{f}_2$ sufficiently close to $f_{2}$ and with a saddle periodic point.

It is clear  from definition that $f_1\x f_2$ is in the closure of $\cU$, and for any $F\in \cU$, there exists a robustly topologically mixing set  homeomorphic to the set $\Lambda \x N$. So its projection to $N$ is surjective.  This completes the proof.
\Endproof


\subsection{Theorem \ref{thm C}: Instabilities in nearly integrable Hamiltonians.}
To prove Theorem \ref{thm C}, it is used that the time one maps of $H_0$ satisfies the hypotheses of Theorem \ref{thm A}. 

\Proofof{Theorem \ref{thm C}} Let $M=\DD^{m}\x\TT^{m}$  and $N=\DD^n\x\TT^{n}$. 

First we perturb the Hamiltonian $h_{1}$ on $M$ to obtain a transversal homoclinic intersection. Since $h_{1}$ satisfies the  {\textsf{A.H.}} condition, then there  is a small perturbation $\tih_{1}$ of $h_{1}$ such that the time one map $\tilde{f}_1$ of its Hamiltonian flow has a nontrivial transitive hyperbolic invariant set. 

Note that the time one map of the Hamiltonian flow of $h_{2}$ is dominated by the restriction of $\tilde{f}_1$ to its hyperbolic basic set.  

Then, we take another small integrable perturbation $\tih_{2}$ of the integrable Hamiltonian $h_{2}$ on $N$  to create a $\delta$-weak hyperbolic periodic point, for $\delta>0$ small enough. 

Note that since the perturbations done in the proof of Theorems \ref{thm A'}, \ref{thm A} and \ref{thm A} are performed by Hamiltonian diffeomorphism and by symplectic changes of coordinates, those theorems and proofs can easily be adapted for Hamiltonians. 

Now we repeat the proof of Theorem \ref{thm A} with the same changes we made in the proof of Theorem \ref{thm B}. Note that here the manifolds $N$ is not compact so we obtain robust transitivity on (strong) continuation of any compact subset of $\Gamma$. This completes the proof of Theorem \ref{thm C}.
\Endproof


\subsection{Theorem \ref{thm D}: Slow-fast systems}

Before giving the proof of  Theorem \ref{thm D}, let us recall a result due to Zehnder and Newhouse. 


\subsubsection{Theorem of Zehnder and Newhouse} \label{s thm ZN}  
Recall that a periodic point $p$ of $f$ of period $n$ is called \textit{quasi-elliptic} 
if $T_pf^n$ has a non-real eigenvalue of norm one, and all eigenvalues of norm one are non-real. Indeed, $C^{r}$ generically every periodic point is either hyperbolic or quasi-elliptic (cf. \cite{rob}, \cite{ne}). Note also that if $f$ is Anosov, then robustly there is no quasi-elliptic periodic point.

\begin{theorem}[{\cite{z, ne}}] \label{thm ZN}  
There is a residual set \biR$~\subset \diff^r_\w(M)$, $1 \leq r \leq \infty$, such that if $f \in$ \biR, then any quasi-elliptic periodic point of $f$ is a limit of transversal homoclinic points of $f$.
\end{theorem}

Moreover, $C^r$ generically, any elliptic point is the limit of 
a sequence of hyperbolic orbits $\{p_n\}$ such that $p_n$ is $\delta_n$-weak hyperbolic point where $\delta_n \to 0$ as $n\to \infty$ (cf. \cite{ne} and \cite{bk}).
Thus, $C^r$ generically, existence of an elliptic periodic point implies the existence of a $\delta$-weak hyperbolic periodic points with arbitrarily small $\delta>0$.

\Proofof{Theorem \ref{thm D}}
In order to prove Theorem \ref{thm D}, it is enough to show that for a sufficiently small $\veps>0$, $H_\veps$ is approximated by a sum of two Hamiltonians whose time one maps satisfy the hypotheses of Theorem \ref{thm A}.

By assumption, $h_1$ satisfies the {\textsf{H.S.}} condition. So, the time one map $f_1$ of  its Hamiltonian flow has a hyperbolic basic set $\Lambda$.
Then, there exists $\veps_0>0$ such that for any $\veps\in(0, \veps_0)$,  the time one map of  the Hamiltonian flow of $\veps h_2$, which we denote it by $f_2$, is dominated by $f_1|_\Lambda$.
Since, the domination is an open property, the same holds for any small perturbation of $h_2$. 

By assumption, the time one map of $h_2$ has an elliptic periodic point, and so does time one map of any small perturbation of $h_2$. On the other hand,
it follows from Section \ref{s thm ZN}  that the time one map of the Hamiltonian flow of any generic perturbation $\tilde{h}_2$ of $h_2$, has an elliptic periodic point which is the limit of a sequence of $\delta$-weak hyperbolic periodic points for arbitrarily small   $\delta>0$.

Now, we choose $\veps=1/k$, where $k\in \NN\cap  (1/\veps_0, \infty)$. Then, $\tilde{f}_2$, the time one map of the Hamiltonian flow of $\veps\tilde{h}_2$,  has an elliptic periodic point which is the limit of a sequence of $\delta$-weak hyperbolic periodic points, for arbitrarily small   $\delta>0$. 

Therefore, $f_1$ and $\tilde{f}_2$ satisfy the hypothesis of Theorem \ref{thm A'}. Thus, there exists a smooth arc of Hamiltonians beginning on $h_1+\veps\tilde{h}_2$ such that the corresponding time one maps verify the conclusion of  Theorem \ref{thm A'}.

Let $\cH$ be the union of all such smooth arcs as $\tilde{h}_2$ tends to $h_2$. Clearly, $H_\veps$ belongs to the $C^\infty$ closure of $\cH$. 

To finish the proof, we will show that for any $\alpha>0$, there is a  
open neighborhood $\cV$ of $\cH$ that verifies (1) and (2) in thesis of the theorem.

Let $H\in\cH$, then its corresponding time one map verifies all the properties of some $F_\mu$ in the proof of  Theorem \ref{thm A'}.
Therefore, it is enough to show that the continuation of the set $\Gamma$ (in the proof of Theorem \ref{thm A'}) satisfies (1), whenever the neighborhood $\cV$ is sufficiently small.  This is a direct consequence of the following:

Let $\cR_{\mu,L}$ be  the set of points in $\cR_\mu$ such that their strong stable and unstable manifolds of diameter $L$ are connected to the blender. Then,  the same holds for the continuation of this set for any nearby system.

Using the remark \ref{rmk Leb}, the set $\cR_\mu = \bigcup_{L>0} \cR_{\mu,L} = \bigcup_{L>0} \Gamma_{F_\mu, L}\cap (\{\hat p\}\x N)$ is not open and dense  but has total Lebesgue measure in 
$\{\hat p\}\x N$. So, for $L>0$ large enough,  $ {\rm vol}(\cR_{\mu,L})> {\rm vol}(N)-\alpha/3$.

In addition, the normally hyperbolic manifold $\{\hat p\}\x N$ and its continuations are $C^r$ close. Consequently, the volume of continuation of $\cR_{\mu,L}$  is close to $ {\rm vol}(\cR_{\mu,L})$. So its projection to $N$ has volume 
$> {\rm vol}(N)-\alpha$. 

Now, let $\Upsilon$ be the continuation of $\Gamma_{F_\mu, L}$. So it is robustly topologically mixing and  contains the continuation of $\cR_{\mu,L}\subset \Gamma_{F_\mu, L}$.
Therefore, the projection of $\Upsilon$  on  $N$ has volume $> {\rm vol}(N)-\alpha$.
This completes the proof of Theorem \ref{thm D}.
\Endproof


\subsection{Corollary \ref{cor ergodic}: Existence of ergodic measures}
\Proofof{Corollary \ref{cor ergodic}}
In the proof of Theorem  \ref{thm A'} we prove that the large transitive set $\Gamma_G$ is in fact contained in the homoclinic class of some hyperbolic periodic point.
In \cite[Theorem 3.1]{abc} it is proved that any homoclinic class coincides with  the support of some ergodic measure (with zero entropy). 
Consequently, the set $\Gamma_G$ is contained  in the support of some ergodic measure.
\Endproof

It seems interesting that the support of ergodic measures varies (lower-semi) continuously in Hausdorff topology as the diffeomorphisms vary in $C^r$ topology. 
This phenomenon suggest a notion of ``stable ergodicity'', weaker that the usual  one (see also Section \ref{s ergodic}).

\subsection{Dichotomy: wandering and instability versus recurrence and transitivity.} \label{sec inst vs. rec}
In this section we are discussing the following dichotomy:
{\it Suppose that the assumption (b) in Theorem \ref{thm A'} fails.
Then, either there exist  large robustly transitive sets as in Theorem \ref{thm A'} or there exist  wandering orbits converging to infinity.}\\

Before proving it,  we need to state some lemmas.

\begin{lemma} \label{lem rec}
There is a residual subset $\cR$ of $~\inter(\Omega(f))$ such that any point in $\cR$ is a (positively and negatively) recurrent point.
\end{lemma}
\Proof
Let $\mathfrak{B}= \{U_i: ~i\in \NN\}$ be a countable topological base in $~\inter(\Omega(f))$.
For every $i \in \NN$, there is $n_i\in \NN$ such that $f^{n_i}(U_i)\cap U_i=\varnothing$. Let $x_i\in V_i:=f^{-n_i}(U_i)\cap U_i$. Since $\mathfrak{B}_k= \{U_i: ~ i \geq k \}$ is also a topological base, the set $\{x_i\}^{\infty}_{r=k}$ is dense in $~\inter(\Omega(f))$. So $\bigcup^{\infty}_{i=k}V_i$ is open and dense subset of $~\inter(\Omega(f))$. Then, $\cR^+:=\bigcap^{\infty}_{k=1} \bigcup^{\infty}_{i=k}V_i$ is residual. 
We claim that $\cR^+ \subset {\rm Rec}^+(f)$. Since  $\mathfrak{B}$  is a topological base, for any $\epsilon>0$ there is a $k_{\epsilon}$ such that, if $i>k_{\epsilon}$ then $\diam (U_i)<\epsilon$.
Now, for any $x\in \cR^+$ and for $i > k_{\epsilon}$, $x \in V_i$. So there is $n_i\in \NN$ such that, $d(f^{n_i}(x), x)< \diam (U_i) <\epsilon$. Since $\epsilon>0$ was arbitrary, this implies that $x$ is a positively recurrent point. We could do it for $f^{-1}$ to obtain a residual subset $\cR^-$ of negatively recurrent points. Any point in the residual set $\cR=\cR^- \cap \cR^+$ is positively and negatively recurrent. 
\Endproof

\begin{remark}
It follows from this lemma that if the
non-wandering set of a diffeomorphism has (large) non-empty interior, as in the case of Theorem  \ref{thm rec}, there is
an iterated function system of its nearby systems exhibiting transitivity in the interior of the non-wandering set. 
\end{remark}

We say that a point $x$ {\it converges to infinity} if for any compact set $U$ there is a number $n_{0}$ such that for any $n>n_{0}$, $f^{n}(x)\notin U$.

The following lemma is a corollary of a variation of Poincar\'e Recurrence Theorem for unbounded measures (due to Hopf, cf.  \cite{mn2}) which yields that for conservative homeomorphisms on manifolds with unbounded measure, almost all points either  are recurrent or converge to infinity.

\begin{lemma} \label{lem wander}
Let $f$  be a conservative homeomorphism on a non-compact manifold with unbounded volume. Then almost all points in  $\Omega(f)^{\complement}$ converge to infinity, in the future and also past iterations.
\end{lemma}

\Proofof{the dichotomy} 
It follows from Lemma \ref{lem wander} that almost all points in $N\setminus\Omega(f_2)$ converge to infinity (for future and past iterations).  Assume that $\Omega(f_2)=N$, but for some  $\tilde{f}_2$ close to $f_2$, $\Omega(\tilde{f}_2) \subsetneq N$. 

We show that the same statement about
transitivity and topologically mixing as the one in Theorem \ref{thm A'} (or Theorem \ref{thm A}) holds in the interior of $\Omega(G)$ for any $G$ close to the constructed family $F_\mu$. Indeed, we use the hypothesis $\Omega(\tilde{f}_2)= N$ only in the last step
of the proof of Theorem \ref{thm A'}, where we apply Lemma \ref{lem blender ph 2} which requires recurrence.

So we may repeat the proof of Theorem \ref{thm A'} to construct the family $F_\mu$. 
Let us fix $L>0$ large enough and  define $\Gamma_{F_\mu, L}$ as in Section \ref{s robust thm A}.
For a $G$ close to $F_\mu$, we may apply Theorem \ref{thm hps}, and so $\Gamma_G$, the continuation of $\Lambda \x N$, is well defined. As in Section \ref{s robust thm A}, let $\Gamma_{G, L}$  be the set defined by the set of trajectories  with strong stable and strong unstable manifolds of diameter $L$  connected to the blender. By the continuity of the strong stable and unstable foliations we see that $\Gamma_{G, L}$  is close to $\Gamma_{F_\mu, L}$ in the Hausdorff topology. 

Let $\tilde\Gamma_L$ be the interior of $\Gamma_{G,L} \cap \Omega(G)$ in $\Gamma_G $. We show that  $\tilde\Gamma_L$ is (robustly) topologically mixing.

If  $\tilde\Gamma_L = \Gamma_{G,L},$ then we may apply Lemma \ref{lem rec} for $G$ restricted to the continuation of $\{\hat{p}\}\x N$, and the proof of  Theorem \ref{thm A'} without any change can be repeated, obtaining therefore the transitivity (and topological mixing) of $\Gamma_G$.

Otherwise, assume that  $\tilde\Gamma_L \subsetneq \Gamma_{G,L}$. Now, recall that the periodic center fibers  are dense in $\Gamma_G$.
For any periodic fiber $\tN_x$ of period $j$ in $\Gamma_G$, almost all points in the interior of $\Omega(G^j|_{\tN_x})$ in $\tN_x$ are recurrent. So $\Gamma_{G, L} \bigcap \inter(\Omega(G^j|_{\tN_x}))$ is contained in the homoclinic class of the periodic point $P_G$. This implies that $\tilde\Gamma_L$ is contained in the homoclinic class of the periodic point $P_G$ and so it is topologically mixing. 
This means  in particular that topologically mixing property holds in the interior of non-wandering set of $G$.

 In fact, we have  proved a stronger statement: it holds in the interior of $\Omega(G)$ in $\Gamma_G$.  On the other hand, almost all points in the complement of these sets converge to infinity. 
This completes the proof of the dichotomy.
\Endproof


\section{Some remarks and open problems} \label{sec remarks}
 Several natural questions arise from  the main results of this paper. Here we mention just few  of them.


\subsection{Are hypotheses of Theorem \ref{thm A}  optimal?}\label{s hypo thm A}
As we mentioned after the statement of Theorem \ref{thm A} we wonder if its hypotheses are optimal. These speculations are based in several results proved in the $C^1$ topology:

\begin{enumerate}
\item It is has been shown that any robustly transitive invariant sets of symplectic diffeomorphisms on compact manifolds are partially  hyperbolic \cite{bdp, dpu, ht, sx}. This means that, the hypothesis (b) is not avoidable to obtain robustness of transitivity (see also Section \ref{s problem ph}).

\item  It has been conjectured that any symplectic diffeomorphism is $C^r$ approximated either by an Anosov or by a diffeomorphism with a quasi-elliptic periodic point. The case $r=1$ has been proved by Newhouse \cite{ne}.
On the other hand,  from Theorem \ref{thm ZN} it follows that any diffeomorphism having a quasi-elliptic periodic point is $C^r$ approximated  by one exhibiting a transversal homoclinic point. Consequently, a $C^1$ generic diffeomorphism $f_1$ satisfies the hypothesis (a).  
Moreover, if $\dim(N)=2$, a generic diffeomorphism $f_2$ either satisfies the hypothesis (c) or  is an Anosov system and so the conclusions of Theorem \ref{thm A} follow. This means that, up to the mentioned conjecture, for generic systems the only essential assumption in Theorem \ref{thm A} is the hypothesis (b).

\end{enumerate}


\subsection{Transitivity and partial hyperbolicity}\label{s problem ph}
The first question concerns the genericity of  robustly mixing partially
hyperbolic sets. Theorem \ref{thm A} suggests that the answer of the following
problem would be positive. 
\begin{problem} Does there exist a residual set \biR$~\subset \diff^r_\w(M)$, $1 \leq r \leq \infty$, such that if $f \in$ \biR, then any normally hyperbolic invariant submanifold $N$ for $f$ with transversal intersection between its stable and unstable manifolds is topologically mixing, provided that $N \subset \inter(\Omega(f))$?
\end{problem}

As in the case of $C^1$ topology (see  \cite{dpu}, \cite{bdp} and  \cite{ht}), we believe that the partial hyperbolicity condition is {\it necessary} for  $C^r$ robustness of mixing. 

\begin{problem}
Let $(M,\w)$ be a symplectic manifold. Suppose that $\Gamma$ is a robustly topologically mixing invariant set for $f$ in $\diff^r_\w(M)$. Is it a partially hyperbolic set? 
\end{problem}
The previous  problem is also related to the $C^r$ stability conjecture which is still open for $r\geq 2$.


\subsection{Ergodicity and stable ergodicity}\label{s ergodic}
In the context of Theorem \ref{thm A}, it follows that for any symplectic diffeomorphism $G$ close to $f_{1}\x f_{2}$, there is a natural invariant  measure $\rho_{_{G}}$ supported on the continuation of $\Lambda \x N$, which  has the form of skew product of the volume on the central fibers over the Bernoulli measure on shift.
So, one may ask about the ergodicity of this measure:

\begin{problem} \label{quest erg}
Is it possible to approximate the product $f_1\x f_2$ of Theorem 
\ref{thm A} in the  $C^{\infty}$ topology by a symplectic diffeomorphism $G$ for which the invariant measure
$\rho_{_{G}}$ supported on the continuation of $\Lambda\x N$ is ergodic or stably
ergodic?  
\end{problem}

In most works about stable ergodicity, a main step is proving (stable) accessibility (cf.  \cite{psh}, \cite{bw}).  Accessibility, roughly speaking, means that any pair of point of a partially hyperbolic set are jointed by a piecewise smooth arc that locally lies in the stable or unstable  manifolds of its points. 
The ``global'' partial hyperbolicity is a fundamental assumption in all results involving  ``stable'' accessibility. 
So, Problem  \ref{quest erg} asks for  new methods. More generally, it highlights the need of a theory for non-global stable ergodicity.


\subsection{Iterated function systems} Motivated  by results in Section \ref{s ifs rec}, it is possible to formulate a series of problems for generic IFS.

\begin{problem} \label{transitive}
Is the iterated function system of any  $C^{r}$-generic pair  of symplectic diffeomorphisms robustly transitive (or topologically mixing, robustly  topologically mixing, minimal)?
\end{problem}

Some progress on this problem has been done very recently in \cite{kn} for the case of surface diffeomorphisms.  

\begin{problem} \label{ergodic}
Is the iterated function system of any  $C^{r}$-generic pair of conservative or symplectic diffeomorphisms ergodic (or  stably ergodic, mixing, etc.)?
\end{problem}

Results in this direction would be very helpful to study the dynamics of certain partially hyperbolic sets. 


\subsection{Other contexts.}\label{applications} We believe that the methods developed here could be applied in other  context, someone with mechanical meanings. We list a few ones.

 {\em Analytic symplectic and Hamiltonian systems}. The perturbations involved in the present paper are typically $C^r$ for any $r\geq 1$. However, we believe that some of them could be performed in the analytic context, specially the ones involving the creation of blenders.

{\em The dynamics near the (quasi) elliptic periodic points in dimensions $\geq 4$}. The dynamic near elliptic periodic points could display transversal homoclinic intersections associated to a hyperbolic periodic points (with a hyperbolic part dominating a $\delta-$weaker hyperbolic one). This situation is similar to the once in Theorem \ref{thm A}.

{\em  Perturbations of time independent Hamiltonian systems by periodic potentials.} Observe that in Theorem \ref{thm C}, the second Hamiltonian is time independent and it is perturbed by a time periodic one. It is natural to wonder if the same approach can be performed for the case that the perturbation is given by a time periodic potential.

{\em  Generic energy levels of time independent Hamiltonian systems}. Observe that some of the Hamiltonians involved in Theorem \ref{thm C} 
 are time periodic. The goal, therefore, would be to extend some of the present results to the time independent context.

{\em Specific mechanical problems}.  Of course, it would be essential to try to apply the present theorems to a concrete mechanical problem.  We wonder if improvement  of our results can be useful in these mechanical models with the goal to obtain robust transitivity.   Of course, a major challenge would be to apply the present approach to the context of the restricted 3-body problem.

{\em  Perturbations of geodesic flows on surfaces by periodic potentials.} It is natural to try to extend the  results about robust transitivity to the case of a metric perturbed by a time periodic potential.

{\em Geodesic flows on manifolds of dimensions larger that two}. A natural context for this problem is to consider warped product (close to a direct product) between a metric of negative curvature with  the  spherical metric. Locally, the new metric resembles the product of a hyperbolic system with an integrable one.

\References{99999}

\bibitem[ABC]{abc}
\name{F. Abdenur, C. Bonatti, S. Crovisier},
Nonuniform hyperbolicity for $C^1$-generic diffeomorphisms 
\jname{Israel J. Math.} {\bf 183} (2011) 1--60.

\bibitem[ArBC]{arbc}
\name{M.-C. Arnaud, C. Bonatti, S. Crovisier},
Dynamiques symplectiques g\'en\'eriques. 
\jname{Ergodic Theory Dynam. Systems}  {\bf 25}  (2005),  no. 5, 1401--1436.

\bibitem[A]{la}
\name{L. Arnold}, \jname{Random dynamical systems}. Springer Monographs in  Mathematics. Springer-Verlag, Berlin, 1998. xvi+586 pp.

\bibitem[A1]{a1} 
\name{V.~I. Arnold},
Small denominators and problems of stability of motion in classical and celestial mechanics, 
\jname{ Uspehi Mat. Nauk} {\bf 18} (1963), no.~6 (114), 91--192

\bibitem[A2]{a2}
\name{V.~I. Arnold},
Instability of dynamical systems with several degrees of freedom, 
\jname{Dokl. Akad. Nauk SSSR} {\bf 156} (1964), 9--12.

\bibitem[AKN]{akn}
\name{V.~I. Arnold, V.V. Kozlov, A.I. Neishtadt},
\jname{Mathematical aspects of classical and celestial mechanics},
Springer Verlag, 2006.

\bibitem[BG]{bg} \name{N. Berglund, B. Gentz},
\jname{Noise-induced phenomena in slow-fast dynamical systems: a sample-path approach},
Birkhauser, 2005.

\bibitem[BK]{bk}
\name{D.~Bernstein and A.~Katok},
Birkhoff periodic orbits for small perturbations of completely
  integrable {H}amiltonian systems with convex {H}amiltonians.
\jname{Invent. Math.}, \textbf{88} (1987), no. 2, 225--241.

\bibitem[BC]{bc}
\name{C. Bonatti, S. Crovisier},
R\'ecurrence et g\'en\'ericit\'e. 
\jname{Invent. Math. }  {\bf 158}  (2004),  no. 1, 33--104.

\bibitem[BD]{bd} 
\name{ C.~Bonatti, L.~J. D{\'\i}az},
Persistence nonhyperbolic transitive diffeomorphisms,
\jname{Annals of Math.} (2), {\bf 143} (1996) 357--396.

\bibitem[BD2]{bd08} 
\name{ C.~Bonatti, L.~J. D{\'\i}az},
Robust heterodimensional cycles and $C^1$-generic dynamics, 
\jname{J. Inst. Math. Jussieu}, {\bf 7}, (2008),   Issue 3, 469--525.

\bibitem[BDP]{bdp}
\name{C.~Bonatti, L.~J. D{\'\i}az, E. Pujals},
A $C^1$-generic dichotomy for diffeomorphisms: Weak forms of hyperbolicity or infinitely many sinks or sources, 
\jname{Annals of Math.} (2), {\bf 158}  (2003),  no. 2, 355--418.

\bibitem[BDV]{bdv} 
\name{ C.~Bonatti, L.~J. D{\'\i}az, M.~Viana},
\jname{Dynamics beyond uniform hyperbolicity},
Encyclopedia of Mathematical Sciences, Springer Verlag, 2004.

\bibitem[BW]{bw} 
\name{K. Burns, A. Wilkinson}. On the ergodicity of partially hyperbolic systems. 
\jname{Annals  of Math.} (1), {\bf 171} (2010),  451--489.

\bibitem[CY]{cy}
\name{C. -Q. Cheng, J. Yan}, 
Existence of Diffusion Orbits in a priori Unstable Hamiltonian Systems,
\jname{J. Diff. Geom.}  {\bf 67} (2004)  

\bibitem[DGLS]{dgls} 
\name{ A. Delshams, M. Gidea,  R. de la Llave, T.M. Seara},
Geometric approaches to the problem of instability in Hamiltonian systems. An informal presentation.  \jname{Hamiltonian dynamical systems and applications},  285--336, NATO Sci. Peace Secur. Ser. B Phys. Biophys., Springer, Dordrecht, 2008. 

\bibitem[DLS]{dls} 
\name{ A. Delshams, R. de la Llave, T.M. Seara},
A geometric mechanism for diffusion in Hamiltonian systems overcoming the large gap problem: heuristics and rigorous verification on a model,
\jname{Mem. Amer. Math. Soc.}  \textbf{179}  (2006),  no. 844, viii+141 pp.

\bibitem[DLS2]{dls06} 
\name{ A. Delshams, R. de la Llave, T.M. Seara},
Orbits of unbounded energy in quasi-periodic perturbations of geodesic flows. (English summary)
\jname{Adv. Math.} \textbf{202} (2006), no. 1, 64--188. 

\bibitem[DPU]{dpu}
\name{L.~J. D{\'\i}az, E.~R. Pujals, R. Ures},
Partial hyperbolicity and robust transitivity,
\jname{Acta Math.} 183  (1999),  no. 1, 1--43.

\bibitem[D]{du}
\name{R. Douady},
Stabilit\'e ou instabilit\'e des points fixes elliptiques,
\jname{Ann. Sci. \'Ecole Norm. Sup. (4)}, \textbf{21} (1988) no. 1, 1--46.

\bibitem[GHS]{GHS} 
\name{F.H. Ghane, A.J. Homburg, A. Sarizadeh}
$C^1$-robustly minimal iterated function systems, 
\jname{Stoch. Dyn. 10 (2010), n0. 1, 155--160}

\bibitem[HPS]{hps} 
\name{M. Hirsch, C. Pugh, M. Shub},
\jname{Invariant manifolds},
Lecture Notes in Math. 583, Springer, Berlin, 1977.

\bibitem[HT]{ht}
\name{V. Horita, A. Tahzibi}, Partial hyperbolicity for symplectic diffeomorphisms. \jname{ Ann. Inst. H. Poincar\'e Anal. Non Lin\'eaire} \textbf{ 23}  (2006),  no. 5, 641--661.

\bibitem[KL]{kl} \name{V. Kaloshin, M. Levi}  An example of Arnold diffusion for near-integrable Hamiltonians,
\jname{Bull. Amer. Math. Soc.} \textbf{45}  (2008), 409-427. 
 
\bibitem[KMV]{klm} 
\name{V. Kaloshin, J. Mather, E. Valdinoci},
Instability of resonant totally elliptic points of symplectic maps in dimension 4, 
\jname{Asterisque},  No. 297 (2004), 79--116.

\bibitem[K]{kat}
\name{A.~Katok},
Lyapunov exponents, entropy and periodic orbits for diffeomorphisms.
\jname{Publ. Math. Inst. Hautes {\'E}tudes Sci.},  \textbf{51} (1980),137--173.

\bibitem[KN]{kn} \name{A. Koropecki, M. Nassiri},  Generic transitivity for area-preserving iterated functions systems. \jname{Math. Z.}, {\bf 266}, (2010), 707--718;  ibid. {\bf 268}, (2011), 601--604.

\bibitem[M\~n]{mn}
\name{R. Ma\~n\'e},
Contributions to the $C^1$-stability conjecture, 
\jname{Topology} {\bf 17} (1978), 386--396.

\bibitem[M\~n2]{mn2}
\name{R. Ma\~n\'e},
Ergodic theory of differentiable dynamics, Springer Verlag, 1987.

\bibitem[MS]{ms}
\name{J.P. Marco, D. Sauzin} 
Stability and instability for Gevrey quasi-convex near-integrable Hamiltonian systems.  
\jname{Publ. Math. Inst. Hautes \'Etudes Sci.}  No. \textbf{96}  (2002), 199--275.

\bibitem[Ma]{ma}
\name{J. Mather},
Arnold Diffusion. I: Announcement of Results,
\jname{Journal of Mathematical Sciences} \textbf{124} (5) : 5275--5289, December 2004.

\bibitem[Mo]{mo}
\name{R. Moeckel},
Generic drift on Cantor sets of annuli,
 in Celestial mechanics (Evanston, IL, 1999), \jname{Contemp. Math.}  {\bf 292} (2002), 163-171.

\bibitem[N]{thesis} 
\name{M. Nassiri},
Robustly Transitive Sets in Nearly Integrable Hamiltonian Systems. 
\jname{Thesis at IMPA (2006).}

\bibitem[Ne]{ne}
\name{S. Newhouse},
Quasi-elliptic periodic points in conservative dynamical systems,
\jname{Amer. J. Math.} \textbf{99} (1977), no. 5, 1061--1087.

\bibitem[PSh]{psh}
\name{C. Pugh, M. Shub},
Stable ergodicity. With an appendix by Alexander Starkov,
\jname{Bull. Amer. Math. Soc. (N.S.)}  \textbf{41}  (2004),  no. 1, 1--41.

\bibitem[PS]{ps}
\name{E.~R. Pujals, M. Sambarino},
Homoclinic bifurcations, dominated splitting, and robust transitivity,
\jname{Handbook of dynamical systems. Vol. 1B},  327--378, Elsevier B. V., Amsterdam, 2006.

\bibitem[R]{rob}
\name{R.~C. Robinson},
Generic properties of conservative systems {I, II}.
\jname{Amer. J. Math.}, \textbf{92} (1970), 562--603, 897--906.

\bibitem[SX]{sx}
\name{R. Saghin, Z. J. Xia } Partial hyperbolicity or dense elliptic periodic points for $C^{1}$-generic symplectic diffeomorphisms.  \jname{Trans. Amer. Math. Soc.}  \textbf{358}  (2006),  no. 11, 5119--5138

\bibitem[Sh]{sh}
\name{M. Shub},
Topologically transitive diffeomorphisms $\TT^4$, 
\jname{Lecture Notes in Math.} {\bf 206}, Springer-Verlag, New York, 1971.

\bibitem[SW]{sw}
\name{M. Shub,  A. Wilkinson},
Stably ergodic approximation: two examples,
\jname{Ergodic Theory Dynam. Systems}  \textbf{20}  (2000),  no. 3, 875--893.

\bibitem[X]{x}
\name{Z. J. Xia},
Arnold diffusion: a variational construction,
Proceedings of the International Congress of Mathematicians, Vol. II (Berlin, 1998).  \jname{Doc. Math. } \textbf{1998},  Extra Vol. II, 867--877.

\bibitem[Z]{z}
\name{E. Zehnder},
Homoclinic points near elliptic fixed points,
\jname{Comm. Pure Appl. Math.} \textbf{26} (1973), 131--182.

\Endrefs
\vspace*{.4cm}

\end{document}